\documentclass[12pt]{amsart}
\usepackage{amsbsy,amssymb,amscd,amsfonts,latexsym,amstext,delarray,
amsmath,graphicx,enumitem} 
\usepackage[all]{xy}

\usepackage{color}

\newtheorem{thm}{Theorem}[section]
\newtheorem{prop}[thm]{Proposition}
\newtheorem{corol}[thm]{Corollary}
\newtheorem{lemma}[thm]{Lemma}
\newtheorem{conj}{Conjecture}[section]
\newtheorem*{conjnono}{Conjecture}
\newtheorem{claim}[thm]{Claim}
{\theoremstyle{definition} \newtheorem{defin}{Definition}[section]}
{\theoremstyle{remark} \newtheorem{remark}{Remark}[section]
\newtheorem{example}{Example}[section]}
\numberwithin{equation}{section}

\def\bA{{\mathbb A}}
\def\Abb{{\mathbb A}}

\def\Bb{{\mathbb B}}

\def\bF{{\mathbb F}}

\def\bL{{\mathbb L}}
\def\Lbb{{\mathbb L}}

\def\Pbb{{\mathbb P}}

\def\Sbb{{\mathbb S}}

\def\Tbb{{\mathbb T}}
\def\bU{{\mathbb U}}
\def\Ubb{{\mathbb U}}

\def\Zbb{{\mathbb Z}}

\def\A{{\mathbb A}}

\renewcommand{\P}{{\mathbb P}}
\def\Q{{\mathbb Q}}

\def\K{{\mathbb K}}

\def\cF{{\mathcal F}}

\def\cM{{\mathcal M}}

\def\cS{{\mathcal S}}

\def\cV{{\mathcal V}}

\newcommand{\qede}{\hfill $\lrcorner$}
\newenvironment{itemized}{\begin{itemize}[itemsep=.5ex, wide]}{\end{itemize}}

\title{Motives of melonic graphs}
\author{Paolo Aluffi, Matilde Marcolli, and Waleed Qaisar}
\address{Florida State University, Tallahassee \\ USA}
\email{aluffi@math.fsu.edu}
\address{California Institute of Technology, Pasadena \\ USA \newline \indent
University of Toronto, Toronto \\ Canada \newline \indent
Perimeter Institute for Theoretical Physics, Waterloo \\ Canada}
\email{matilde@caltech.edu}
\address{University of Toronto, Toronto \\ Canada}
\email{waleed.qaisar@mail.utoronto.ca}

\begin{document}

\begin{abstract}
We investigate recursive relations for the Grothendieck classes of the affine graph hypersurface
complements of melonic graphs. 
We compute these classes explicitly for several families of melonic graphs, focusing 
on the case of graphs with valence-$4$ internal vertices, relevant to CTKT tensor models.
The results hint at a complex and
interesting structure, in terms of divisibility relations or nontrivial relations between classes of
graphs in different families. Using the recursive relations we prove that the Grothendieck
classes of all melonic graphs are positive as polynomials in the class of the moduli space
$\mathcal M_{0,4}$. We also conjecture that the corresponding polynomials are {\em log-concave,\/}
on the basis of hundreds of explicit computations.
\end{abstract}

\maketitle

\tableofcontents

\section{Introduction}

In this paper we obtain a recursive formula for the Grothendieck classes (virtual motives) 
of the graph hypersurfaces associated to the melon-tadpole graphs. This provides a recursively
constructed family of mixed-Tate motives, which includes the motives associated to the leading 
melonic terms of certain bosonic tensor models.

\smallskip

Our motivation in considering the behavior of the motives of melon and melon-tadpole
graphs comes from the fact that several interesting physical models are dominated in
the large $N$ limit by melonic graphs. This is the case for SYK models (see \cite{BoNaTa}
for a rigorous diagrammatic proof), as well as in group field theory (see for instance \cite{BCORS}) 
and tensor models (\cite{BGRR}, \cite{FuTa}, \cite{Gu3}, \cite{GuRy}), 
which include generalizations of the SYK models (see for instance \cite{GroRo}, \cite{Wit}).

\smallskip
\subsection{Graph polynomials and CTKT models}

The graph polynomials that one expects to find, when representing amplitudes in 
Feynman parametric form in the setting of group field theory and tensor models, 
are usually of the form described in \cite{Gu4} or \cite{Tan}. The Tanasa graph polynomials
of \cite{Tan} are generalizations of the Bollob\'as-Riordan polynomial that satisfy the
deletion-contraction relation. Similarly, the Gurau polynomials of \cite{Gu4} 
also satisfy a deletion-contraction relation. The motives of hypersurfaces associated
to these polynomials may be, in principle, amenable to the kind
of algebro-geometric techniques discussed in \cite{AluMa11}, which we rely
on in this paper, but in a form more similar to the case of the Potts models we 
analyzed in \cite{AluMa-Potts}. 
However, the computation of the Grothendieck class we obtain here relies
essentially on the recursive form of the Grothendieck class for splitting an
edge and for replacing an edge by a number of parallel edges, obtained in 
\cite{AluMa11}. These formulas do not have a simple counterpart 
for the case of the Potts models and other graph polynomials with
deletion-contraction. This means that a more general computation 
of the polynomials of \cite{Gu4} or \cite{Tan} probably requires a much
more thorough analysis and would not be an immediate generalization of the
argument presented here. 
Other parametric realizations of tensor models, such as \cite{GelTor}, do not even
satisfy a deletion-contraction relation, hence they cannot be addressed via the
method of \cite{AluMa11} and of this paper.

\smallskip

The case we focus on here, however, is simpler and it involves the usual graph
hypersurfaces associated to the Kirchhoff--Symanzik polynomial of the graph,
for a massless scalar theory. These are relevant to tensor models in the
case of the melonic sector of the CTKT models. We briefly recall below the
setting used in \cite{BeGuHa} that motivates the computations we present in
this paper. The case of the graph polynomials of \cite{Gu4} or \cite{Tan} will
be left to a future investigation. Note that, if a similar argument can be applied
to such polynomials, or to the massive melonic graphs, one does not expect
to obtain a family of motives with the mixed-Tate property, since it is known
that already for small graphs in such families the mixed-Tate property fails,
\cite{BloVan}, \cite{MaTa}. Thus, the mixed-Tate property is
certainly specific to the case of the massless Kirchhoff--Symanzik polynomial.
 
\subsection{CTKT models and melonic Feynman graphs}

We focus here on the modified version of the $O(N)^3$ model of Klebanov and Tarnopolsky \cite{KleTar}
considered in \cite{BeGuHa}, which generalizes the zero-dimensional version of \cite{CaTa}. These models 
are referred to in \cite{BeGuHa} as CTKT models and we will maintain the same terminology here.

\smallskip

We recall the following setting from \cite{BeGuHa}.
One considers a real rank three tensor field $\phi_{{\bf a}}(x)$, with 
${\bf a}=(a_1,a_2,a_3)$ that transforms under $O(N)^3$, with action  functional
\begin{equation}\label{CTKTaction}
\begin{array}{ll}
\cS[\phi]= & \displaystyle{\frac{1}{2} \int \phi_{\bf a}(x)\, (-\Delta) \, \phi_{\bf a}(x) \, d{\rm vol}(x) 
+ \cS^{int}[\phi]}  \\[3mm]
\cS^{int}[\phi]=& \displaystyle{ \frac{m^2}{2} \int \phi_{\bf a}(x)\, \delta_{{\bf ab}}\, \phi_{\bf b}(x) \, d{\rm vol}(x) +} \\[3mm]
& \displaystyle{ \frac{\lambda_t}{4 N^{3/2}} \int \delta^t_{{\bf abcd}}\,  \phi_{\bf a}(x)\,  \phi_{\bf b}(x)\,  \phi_{\bf c}(x)\,  \phi_{\bf d}(x)\,
d{\rm vol}(x) +} \\[3mm]
& \displaystyle{  \int \left(  \frac{\lambda_p}{4 N^2} \delta^p_{{\bf ab};{\bf cd}} + \frac{\lambda_d}{4N^3} \delta^d_{{\bf ab};{\bf cd}} \right) \,  \phi_{\bf a}(x)\,  \phi_{\bf b}(x)\,  \phi_{\bf c}(x)\,  \phi_{\bf d}(x)\,
d{\rm vol}(x) }
\end{array}
\end{equation}
with $\Delta=\partial_\mu \partial^\mu$ and with 
\[ 
\begin{array}{ll} 
\delta^t_{{\bf abcd}}=\delta_{a_1 b_1}\delta_{c_1 d_1} \delta_{a_2 c_2} \delta_{b_2 d_2} \delta_{a_3 d_3} \delta_{b_3 c_3} & \delta_{{\bf ab}}=\prod_i \delta_{a_i,b_i}, \\[2mm]
\delta^p_{{\bf ab};{\bf cd}}=\frac{1}{3} \sum_i \delta_{a_i c_i} \delta_{b_i d_i} \prod_{i\neq j} \delta_{a_j b_j} \delta_{c_j d_j}, & \delta^d_{{\bf ab};{\bf cd}} = \delta_{{\bf ab}} \delta_{{\bf cd}}. \end{array}
\]
The labels $t,p,d$ distinguish the tetrahedron, pillow, and double-trace patterns of contraction. 
When edges are colored (red, green, or blue) according to the values of the tensor indices in $\{ 1,2,3 \}$
these different quartic terms correspond to the graphs of Figure~\ref{tpdFig} 
(with three different choices of the pillow contraction depending on the color of the vertical edge).

\begin{figure}
\begin{center}
\includegraphics[scale=.6]{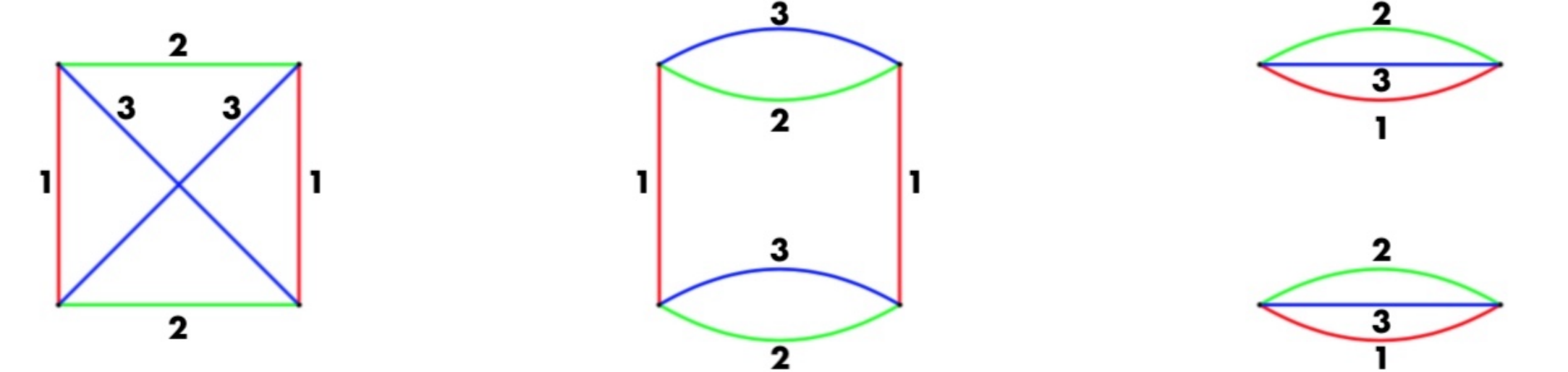}
\caption{Tetrahedron, pillow, and double-trace contractions in CTKT models,  \cite{BeGuHa}.\label{tpdFig}}
\end{center}
\end{figure}

When one computes the contributions to the expansion at leading order in $1/N$ and all
orders in the coupling constants, this is usually done using the $4$-colored graphs expansion
of tensor models (\cite{BGR}, \cite{Gu1}, \cite{Gu2})
with $3$-colored graphs for the different interaction terms as mentioned above (bubbles) 
and another color for the propagators connecting these $3$-colored bubbles.
However, as shown in  \cite{BeGuHa}, it is possible to also consider an expansion
in ordinary Feynman graphs, which are obtained by shrinking all the bubbles to points.
The free energy of the model is written in \cite{BeGuHa} in the form of a sum over 
connected vacuum $4$-colored graphs with labelled tensor vertices,
\[ 
\cF= \sum_G N^{F-\frac{3}{2}n_t -2n_p -3 n_d} \frac{\lambda_t^{n_t}}{n_t ! \, 4^{n_t}}
\, \frac{\lambda_p^{n_p}}{n_p !\, 12^{n_p}} \, \frac{\lambda_d^{n_d}}{n_d ! \, 4^{n_d}}\, 
(-1)^{n_t+n_p+n_d+1} \bA(G),
\]
with $n_t(G)$, $n_p(G)$, and $n_d(G)$ the number of 
tetrahedral, pillow, and double-trace bubbles, respectively, and $F(G)$ the number of faces and with
$\bA(G)$ the amplitude of $G$ written in terms of edge propagators (see \S 2.1 of \cite{BeGuHa}).
One then replaces the $4$-colored graphs $G$ in this expansion with ordinary Feynman graphs
by first replacing all the pillow and double trace bubbles with their minimal resolution in terms of
tetrahedral bubbles (as in Figure~3 of \cite{BeGuHa}). An ordinary Feynman graph is then
obtained by replacing these bubble by vertices. The resulting graph corresponds to a term
of order zero in $1/N$ iff it is a melon-tadpole graph, that is, a graph obtained by
iterated insertion of melons or tadpoles into a melon or tadpole (Figure~\ref{meltadFig}). 
In the absence of pillows and double traces one would obtain just melonic graphs.
The amplitudes $\bA(G)$ of the resulting ordinary melon-tadpole Feynman graphs can
then be computed in the Feynman parametric form, in terms of the Kirchhoff--Symanzik polynomial,
as in \cite{BeGuHa}. We will not discuss here the renormalization problem for the
resulting Feynman integrals, for which we refer the reader to \cite{BeGuHa}. We focus
here instead on the algebro-geometric and motivic properties of these 
melon-tadpole Feynman integrals. 

\begin{figure}
\begin{center}
\includegraphics[scale=0.6]{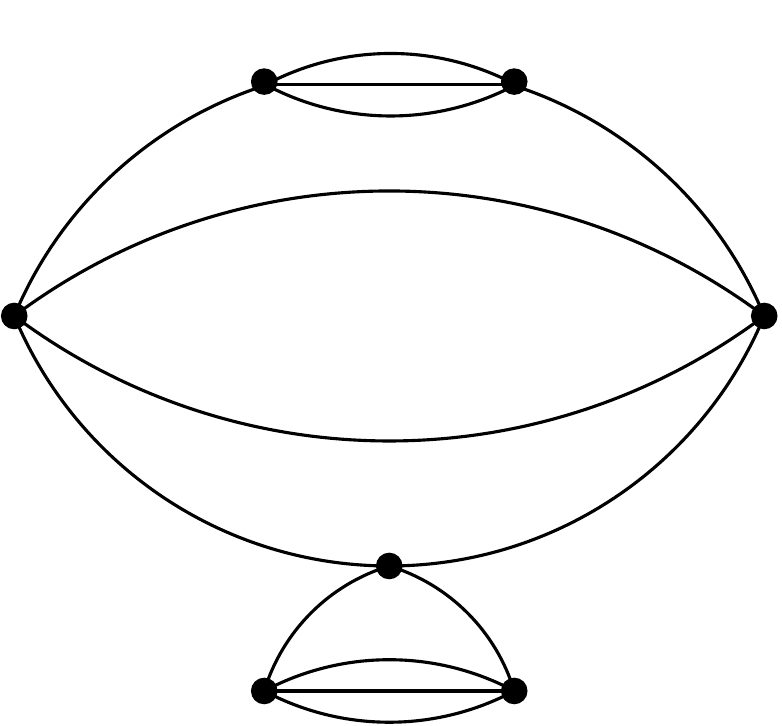}
\caption{Melon-tadpole graphs in CTKT models, \cite{BeGuHa}.\label{meltadFig}}
\end{center}
\end{figure}

\smallskip

{}From the point of view of motivic structures in quantum field theory (see \cite{Mar-book}
for a general overview), our goal here is to show  that massless CTKT models are dominated 
by a recursively constructed family of mixed-Tate motives. 

\smallskip
\subsection{Families of melonic graphs}

The melonic and melon-tadpole graphs that occur in the massless CTKT models
are all constrained by the condition that all vertices have valence $4$, because of
the form \eqref{CTKTaction} of the action functional. In order to study the recursive
properties of the Grothendieck classes associated to these graphs, however, it is
convenient to consider them as a subfamily of a larger family of graphs, which
include melonic graphs with vertices of arbitrary valences. 

\smallskip

Moreover, in the
typical description of melonic graphs, one assumes that the melonic insertions 
are separated by edge propagators (equivalently, one performs an insertion
by first splitting an edge into three edges by the insertion of two valence-two
vertices and then replaced the middle edge by a number of parallel edges).
Again, in our setting it is more convenient to consider these graphs as a 
subfamily of a larger family of melonic graphs where an edge can be
split into a number of subedges and each of them replaced by a set
of parallel edges. The typical case of graphs with only valence-four
internal vertices and including edge propagators will guide us in the choice 
of the examples illustrating the main recursive construction.

\smallskip

We also consider graphs with external edges and graphs (vacuum bubbles)
without any external edges. Instead of following the usual physics convention
of regarding external edges as half-edges (flags), we consider then as edges
with a valence-one vertex. In this setting, when considering non-vacuum
graphs for the CTKT case, we will allow formal valence-one vertices (to mark the
external edges) in addition to the valence-four vertices of the self-coupling
interactions. 

\smallskip

We will not treat separately the melon-tadpole graphs. Indeed our
more general class of graphs includes the operation of bisecting
an edge with an intermediate vertex and a melon-tadpole graph
is simply obtained by grafting together at the vertex two melonic
graphs with this operation performed on one of their edges. Since
the Grothendieck classes for graphs joined at a vertex is just a
product, these classes are easily derived from the ones in the
family we work with. 

\smallskip
\subsection{Summary of the paper}

In \S \ref{BtmNc} we present a convenient formalism for the recursive construction
of melonic graphs with arbitrary valences and we show that a subclass of ``reduced
constructions" always suffices. We reformulate the construction in terms of labelled
bipartite rooted trees. We then focus on the main case of interest for CTKT models,
where graphs have all (internal) vertices of valence four.

\smallskip

In \S\ref{MGc} we recall some basic facts about the Grothendieck ring of
varieties, the parametric Feynman integrals, the graph hypersurfaces
defined by the Kirchhoff--Symanzik polynomial, and the Grothendieck
classes of the affine graph hypersurface complement. 
We focus on the Grothendieck class because it is universal
among invariants which behave well with respect to basic set-theoretic
operations. For example, the Grothendieck class determines the Hodge--Deligne
numbers of the complement of the (affine) graph hypersurface, as well as the
number of points of the complement over finite fields. 
We obtain a recursive formula for the Grothendieck class of melonic graphs
with arbitrary valences. This formula can be effectively implemented in a
standard symbolic manipulation system, and is also useful as a tool to study
general features of Grothendieck classes of melonic graphs. For example,
we prove that the Grothendieck class of a melonic graph can be expressed
as a polynomial with {\em positive\/} coefficients in the class~$\Sbb$ of the
moduli space $\mathcal M_{0,4}$, i.e., $\Sbb=[\P^1\smallsetminus \{0,1,\infty\}]$.
Extensive computer evidence also suggests the following:
\begin{conjnono}
The polynomial expressing the Grothendieck class of a melonic graph is
log-concave (in the sense of~\cite{Stan}).
\end{conjnono}

\smallskip

It is well known (see \cite{Huh}) that the log-concavity property often reflects
some deeper underlying geometric structure, in the form of some kind of
Hodge--de Rham relations. It seems likely that log-concavity of these Grothendieck
classes as functions of $\Sbb$ may indeed be pointing to some richer geometric
structure.

\smallskip

In \S \ref{MEcacI} we focus on the case of melonic graphs with (internal)
vertices of valence four, and we consider particular recursive subfamilies,
for which explicit generating functions can be obtained, both for vacuum
bubbles and for graphs with external edges, with an explicit relations
between these two cases. 
The generating functions in these cases and in the others considered in the
paper were first obtained by carrying out explicit computations using
the recursive formula obtained in~\S\ref{MGc}. As an example of the type
of result we obtain, consider the family consisting of graphs $\Gamma_n$ 
of the form
\begin{center}
\includegraphics[scale=.4]{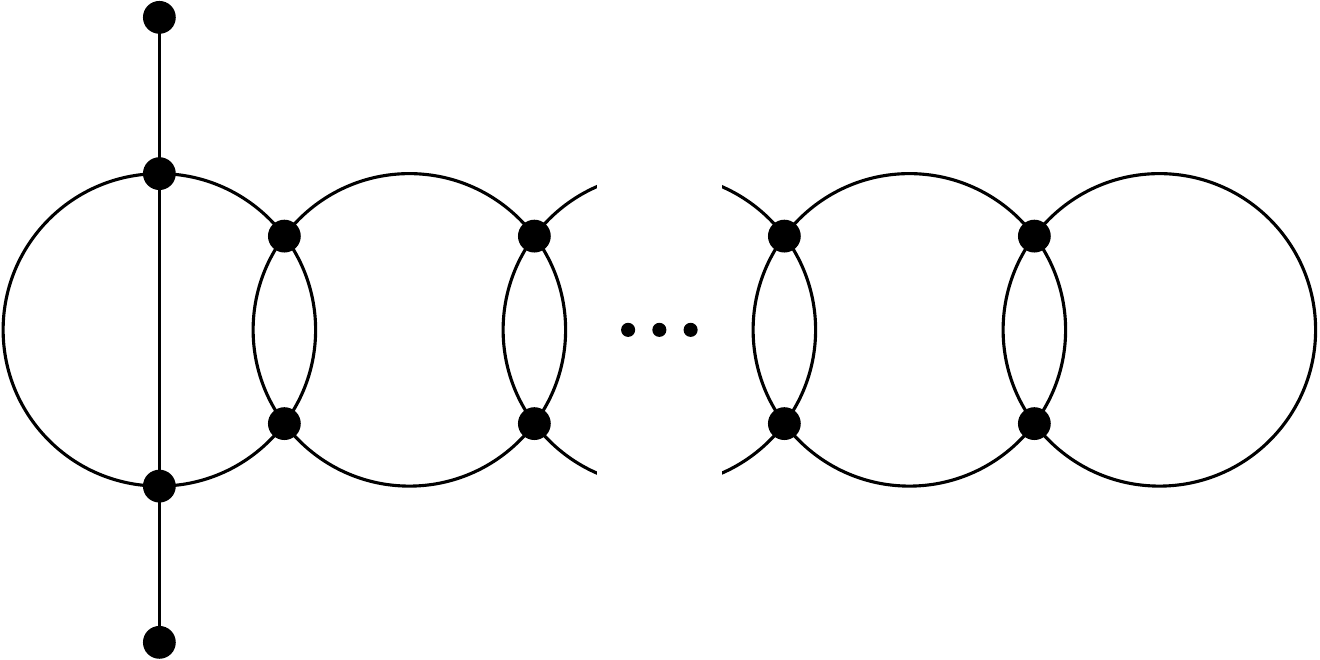}
\end{center}
with $n$ interlocked circles. Let $P_n(u,v)$ be the {\em Hodge--Deligne polynomial\/}
of the complement $Z_n$ of the affine graph hypersurface determined by $\Gamma_n$.
That is, $P_n(u,v)=\sum e^{p,q} u^p v^q$, where 
$e^{p,q}=\sum_k (-1)^k h^{p,q}(H^k_c(Z_n))$ (see e.g., \cite{DaKo}). As a consequence
of Proposition~\ref{prop:GCgamman}, the following holds:
\[
P_n(u,v)=(uv-1)^n (uv)^{2n+1}\cdot A_n(uv-1)\quad,
\]
where the polynomial $A_n(t)$ is determined by the equality of formal power series
$\sum_{n\ge 0} A_n(t) r^n=\sum_{k\ge 0} a_k(r,t)$, with
\[
\sum_{k\ge 0} a_k(r,t)\frac{s^k}{k!}=e^{rs} \cos((r^2-rt)^{\frac 12}s)\quad.
\]
Alternative expressions for $A_n(t)$ are given in~\S\ref{MEcacII}; in fact, 
the information carried by the polynomials $A_n(t)$ may be encoded in a 
{\em rational\/} generating function.

In \S\ref{MEcacII} we analyze from the same viewpoint extensions of these 
recursive subfamilies to the more general case of arbitrary valences. Again
we obtain that the corresponding Grothendieck classes are determined by
rational generating functions.

\smallskip

In \S \ref{Mav} we focus on the melonic vacuum bubbles, and we establish
a general relation between their Grothendieck classes and those of associated
non-vacuum graphs. We describe a procedure for studying the structure of
valence-four melonic vacuum bubbles in terms of their tree structure, and
we identify certain families of recursive relations, in the form of ``melonic
vacuum stars". 

\smallskip

In \S \ref{Mp} we give rigorous proofs of all the statements presented in the
previous sections.

\section{Melonic graphs}\label{BtmNc}

\subsection{The construction of melonic graphs}

A graph with two vertices and $n$ parallel edges connecting them is variously 
referred to in the literature as a melon graph, a banana graph, or a sunset graph.
In the spirit of botanical egalitarianism, we will use the ``banana'' terminology 
when referring to these basic building blocks, and call ``melonic'' the result of
iterating the operation of replacing edges of a graph by strings of bananas.
(We call this operation the `bananification' of the edge.)

Thus, the basic iterative operation constructing melonic graphs is the following:
\begin{center}
\includegraphics[scale=.4]{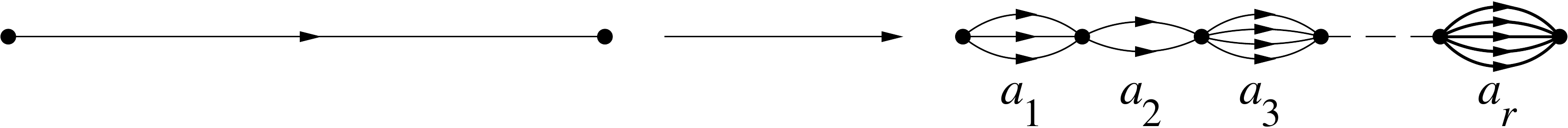}
\end{center}
We allow arbitrary sizes $a_1,\dots, a_r$ for the banana components.
Edges ought to be directed in order to determine the order of inclusion of the bananas;
in fact this will be done implicitly in what follows, since it does not affect the invariant
(Grothendieck class) we are computing.
A melonic graph is obtained by applying this operation to an initial single edge, then 
applying it iteratively to any edge of the resulting graphs.

We can refer to the initial edge as the graph obtained `at stage $0$'; the application of
the iterative process at any stage may be encoded by a tuple
\[
((a_1,\dots, a_r),p,k)
\]
to represent the replacement of one single edge in the $k$-th banana constructed at stage $p$.

\begin{example}\label{ex:illex}
The construction
\begin{center}
\includegraphics[scale=.33]{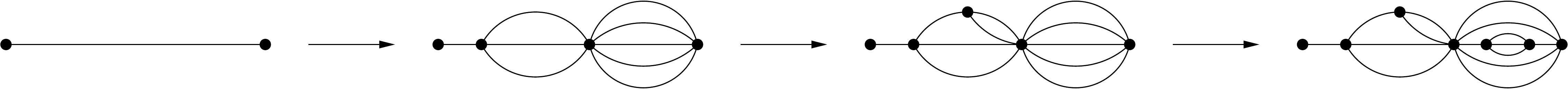}
\end{center}
may be represented by the tuples
\[
((1,3,5),0,1)\quad,\quad ((1,2),1,2) \quad,\quad ((1,3,1),1,3)\quad:
\]
\begin{itemized}
\item the first operation replaces the single edge at stage $0$ with a string consisting of 
a $1$-banana, a $3$ banana, and a $5$ banana; this is stage~$1$;
\item the second operation replaces one edge in the second banana constructed in 
stage~$1$ with a string consisting of a $1$-banana and a $2$-banana; this is stage~$2$;
\item and the third operation replaces one edge in the third banana constructed in 
stage~$1$ with a string consisting of a $1$-banana, a $3$-banana, and a $1$-banana.
This is stage~$3$.
\end{itemized}
Following this sequence of operations with $((2,3),2,1)$ would replace one edge in
the first banana produced at stage $2$ (which actually consists of a single edge)
with a string consisting of a $2$ banana followed by a $3$ banana, producing the 
graph
\begin{center}
\includegraphics[scale=.6]{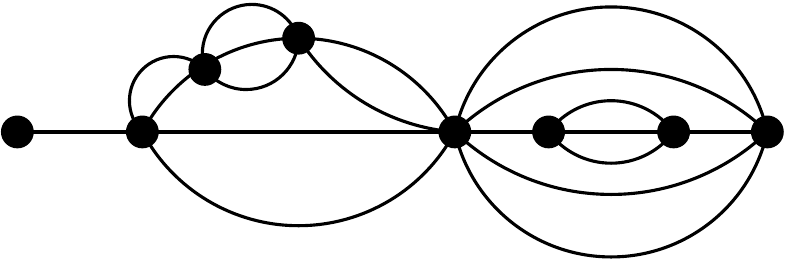}
\end{center}
(As observed below, the same graph admits different constructions.)
\qede\end{example}

Formally, we can make the following definition.

\begin{defin}\label{def:melcon}
An `$n$-stage' (or `depth $n$') {\em melonic construction\/} is a list $T=(t_1,\dots, t_n)$ 
of tuples $t_s=(b_s,p_s,k_s)$ such that
\begin{itemized}
\item[(i)] $b_s=(a_1,\dots a_{r_s})$ is a tuple of positive integers,
of length $|b_s|:=r_s\ge 1$. (Thus, the tuple is non-empty.)
\item[(ii)] $p_s$ is an integer, $0\le p_s<s$;
\item[(iii)] $k_s$ is an integer, $1\le k_s\le |b_{p_s}|$.
\item[(iv)] $p_s>0$ for all $s>1$. (By (ii), $p_1=0$.)
\item[(v)] For all $t_i=((a_1,\dots a_{r_i}),p_i,k_i)$, $i=1,\dots, n$, and all
$j=1,\dots, r_i$, at most $a_j$ tuples $t_s=(b_s,p_s,k_s)$ have $p_s=i,k_s=j$.
\end{itemized}
The length $n$ of the melonic construction is its `depth'.
\qede\end{defin}
The motivation behind these requirements should be evident from the interpretation
discussed above. For example, (v) expresses the constraint that the $j$-th banana
constructed at stage $i$ has enough edges to accommodate later replacements.

\begin{defin}\label{def:melgra}
A {\em melonic graph\/} is a graph determined by a melonic construction by the
procedure explained above.
\qede\end{defin}

Every melonic construction determines a melonic graph up to graph isomorphism.
Of course different melonic constructions may determine the same melonic graph.
We say that two constructions $T'$, $T''$ are `equivalent'
if the resulting graphs are isomorphic.

\begin{example}\label{ex:twoiso}
The construction $\{((2),0,1),((1,3,1),1,1),((1,3,1,1,1),2,2)\}$ determines a melonic
graph as follows:
\begin{center}
\includegraphics[scale=.4]{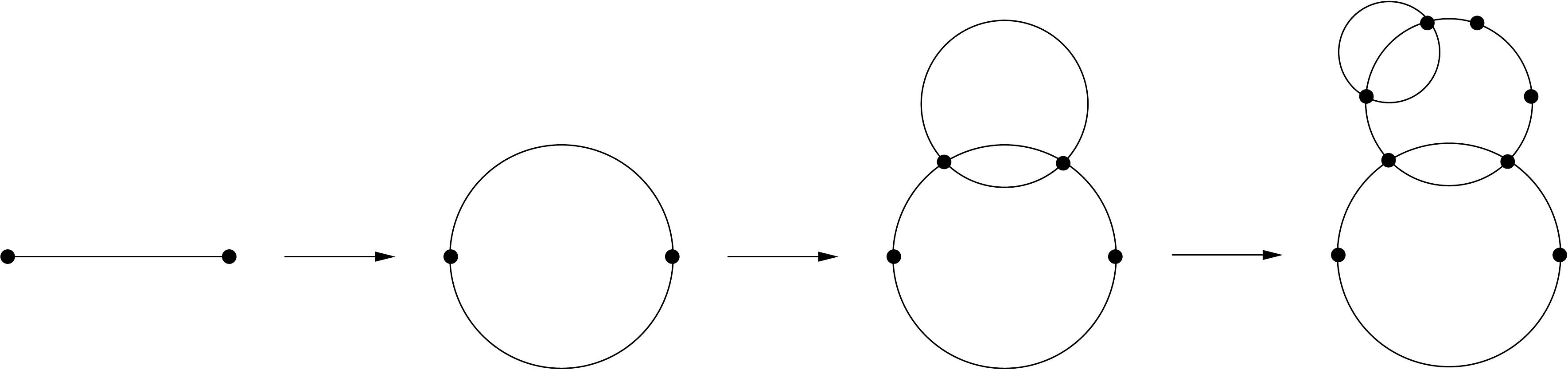}
\end{center}
The construction $\{((2),0,1),((1,3,1,3,1),1,1),((1,1,1),2,4)\}$ produces an isomorphic graph.
\begin{center}
\includegraphics[scale=.4]{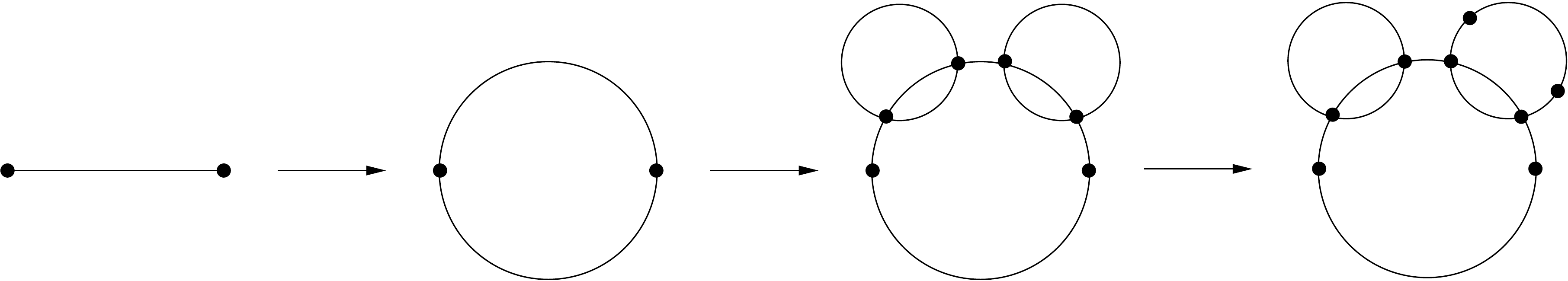}
\end{center}
\qede\end{example}

Also: The graph in Example~\ref{ex:illex} was obtained from the melonic construction
\[
\{((1,3,5),0,1), ((\underline 1,2),1,2), ((1,3,1),1,3), ((\underline{2,3}),2,1)\}\quad.
\]
The same graph can be obtained by the (shorter) construction
\[
\{((1,3,5),0,1), ((2,3,2),1,2), ((1,3,1),1,3)\}\quad.
\]
Similarly, the second construction in Example~\ref{ex:twoiso} produces the same graph
as the (longer) construction
\[
\{((2),0,1),((1,3,\underline 1),1,1),((\underline{1,3,1}),2,3),((1,1,1),3,3)\}\quad.
\]
In both cases, the shorter construction is obtained by implementing the replacement
of the (single) edge in a $1$-banana produced at stage $s$ (underlined) by inserting 
the appropriate tuple (also underlined) directly at stage $s$.

\smallskip
\subsection{Reduced melonic constructions}

Constructions such as Example~\ref{ex:twoiso} suggest the following definition.

\begin{defin}
We say that a construction is {\em reduced\/} if it does not prescribe the replacement
of the edge of a $1$-banana past stage $0$.
\qede\end{defin}

Formally, this requirement prescribes that
\begin{itemized}
\item[(vi)] For all $t_i=((a_1,\dots a_{r_i}),p_i,k_i)$, $i=1,\dots, n$:
If $a_j=1$, then $k_s\ne j$ for all $s$ such that $p_s=i$.
\end{itemized}

The process illustrated above---replacing $1$-bananas by their
descendants---may be performed on every melonic construction, and produces
an equivalent reduced construction. Therefore:

\begin{lemma}\label{lem:redrep}
Every melonic graph admits a reduced construction.
\end{lemma}

Reduced constructions suffice in order to define melonic graphs, but it is important
to consider non-reduced constructions as well; these may appear in intermediate
steps of the recursive computation we will obtain in~\S\ref{MGc}.

\smallskip
\subsection{Melonic graphs and bipartite rooted trees}

There is a convenient way to visualize a melonic construction as a labeled bipartite 
tree. Each tuple $((a_1,\dots, a_r), p, k)$ may be viewed as a rooted tree
\begin{center}
\includegraphics[scale=.5]{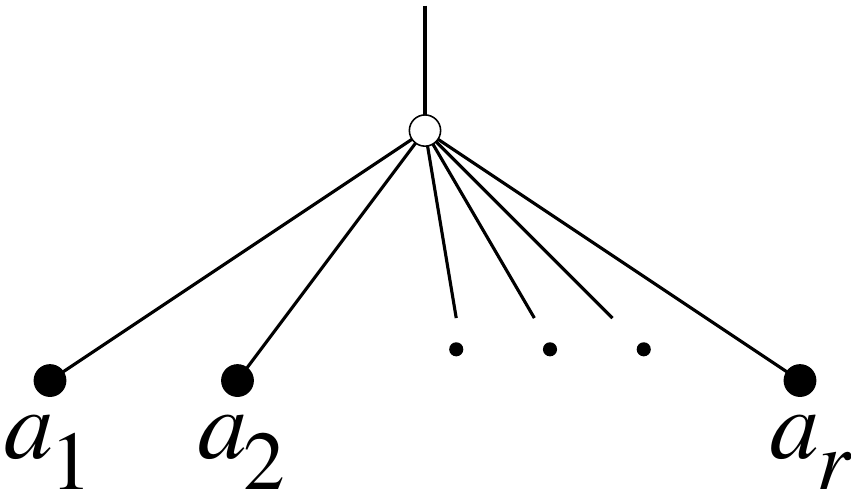}
\end{center}
with (black) leaves labeled by the integers $a_i$. The (white) root will be attached 
to the $k$-th leaf of the $p$-th tree; this grafting procedure builds a rooted tree 
encoding the same information as a melonic construction. Item~(v) in 
Definition~\ref{def:melcon} amounts to the requirement that the valence of a 
(black) node labeled $a$ be at most $a+1$; that is, at most $a$ `descending' 
edges can be adjacent to such a vertex. 

The tree corresponding to a melonic construction has one white node for each
tuple in the construction; thus, the depth of the melonic construction equals the
number of white nodes in the corresponding tree.

\begin{example}
The rooted trees corresponding to the two melonic constructions in 
Example~\ref{ex:twoiso} are
\begin{center}
\includegraphics[scale=.4]{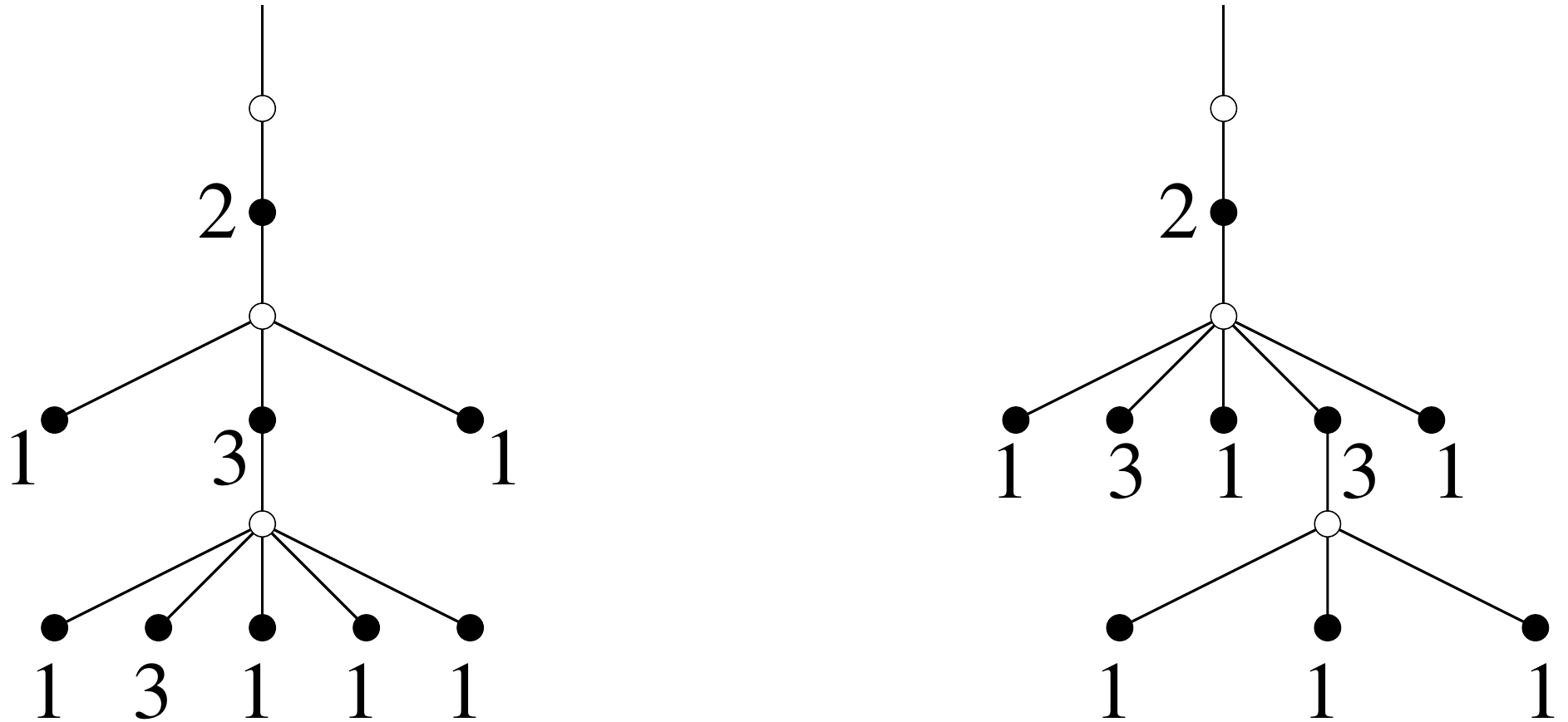}
\end{center}
As noted in Example~\ref{ex:twoiso}, these non-isomorphic labeled trees determine
isomorphic melonic graphs.
\qede\end{example}

The `reduced' condition~(vi) is the requirement that nodes labeled $1$ 
necessarily be leaves. Every tree can be reduced (cf.~Lemma~\ref{lem:redrep}) 
by `sliding up' trees grafted at nodes labeled $1$, as the case encountered in 
Example~\ref{ex:illex} illustrates.
\begin{center}
\includegraphics[scale=.5]{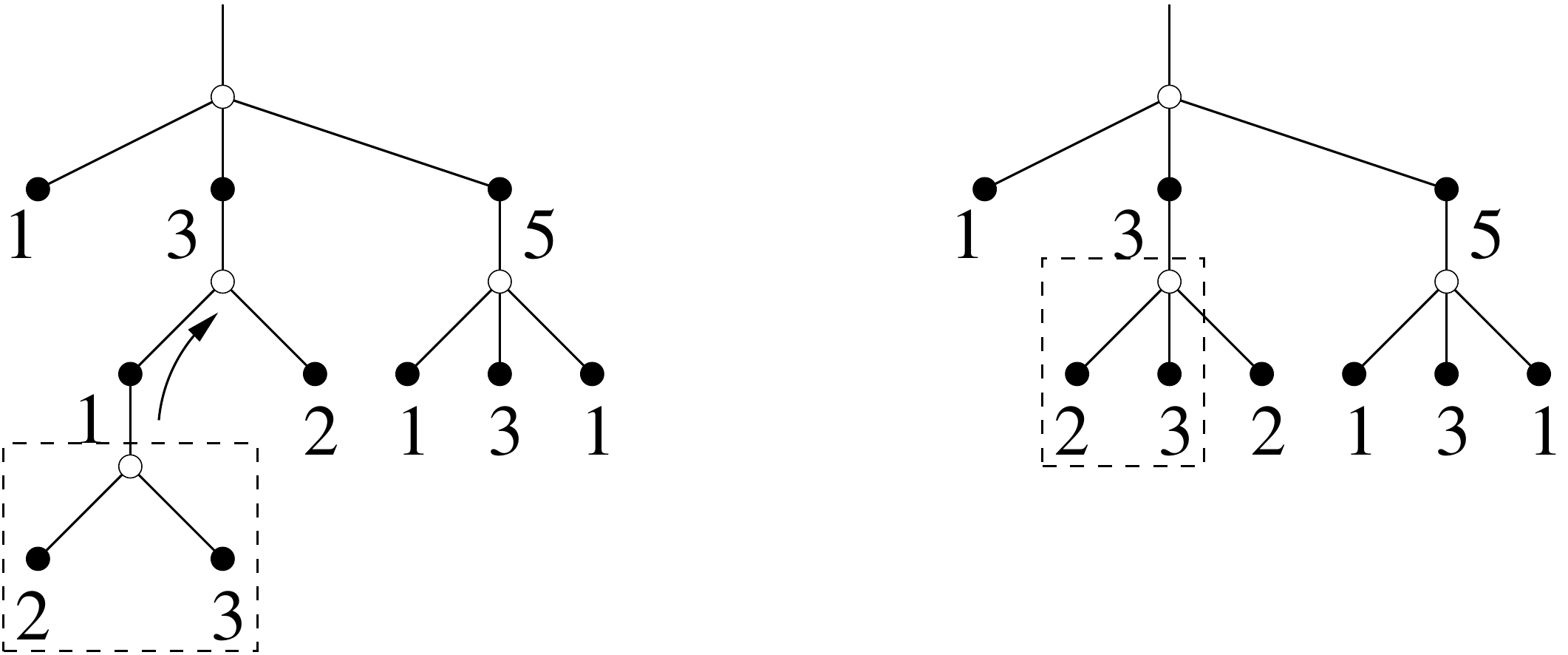}
\end{center}

We also note that the melonic graphs determined in Definition~\ref{def:melgra} have 
arbitrary vertex valences, while in a specific physical theory the valences are 
constrained by the terms in the action. The additional generality is needed for the
recursion formula we will obtain in~\S\ref{MGc}; we will choose families of graphs
with fixed valence in most of the examples illustrating the recursion in~\S\ref{MEcacI},
\ref{MEcacII}, and~\ref{Mav}.

\smallskip
\subsection{Valence-four melonic graphs}

We will be especially interested in the case in which the valence of all internal
vertices of the melonic graph is~$4$. The corresponding melonic constructions
consist of tuples of the type
\[
t_s=((1,3,1),p_s,k_s)
\]
where $k_s=1, 2$, or $3$. The building blocks of these graphs are
\begin{center}
\includegraphics[scale=.5]{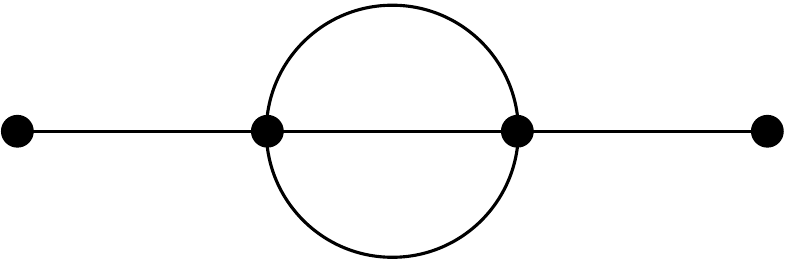}
\end{center}

Up to equivalence, a melonic construction $(t_1,\dots, t_m)$ with 
$t_s=((1,3,1),p_s,k_s)$ as above is determined by the tuple 
$(0,p_2^\pm,\dots, p_n^\pm)$, where each $p_s$ for $s>1$ is marked as $p_s^+$ if
$k_s=2$ and $p_s^-$ if $k_s=1$ or $3$. For example, $(0,1^+,2^+,3^+,4^+)$
indicates that at each stage the new splitting $(1,3,1)$ is performed on one of the
$3$ parallel edges at the previous stage. The corresponding melonic graph may
be drawn as follows:
\begin{center}
\includegraphics[scale=.4]{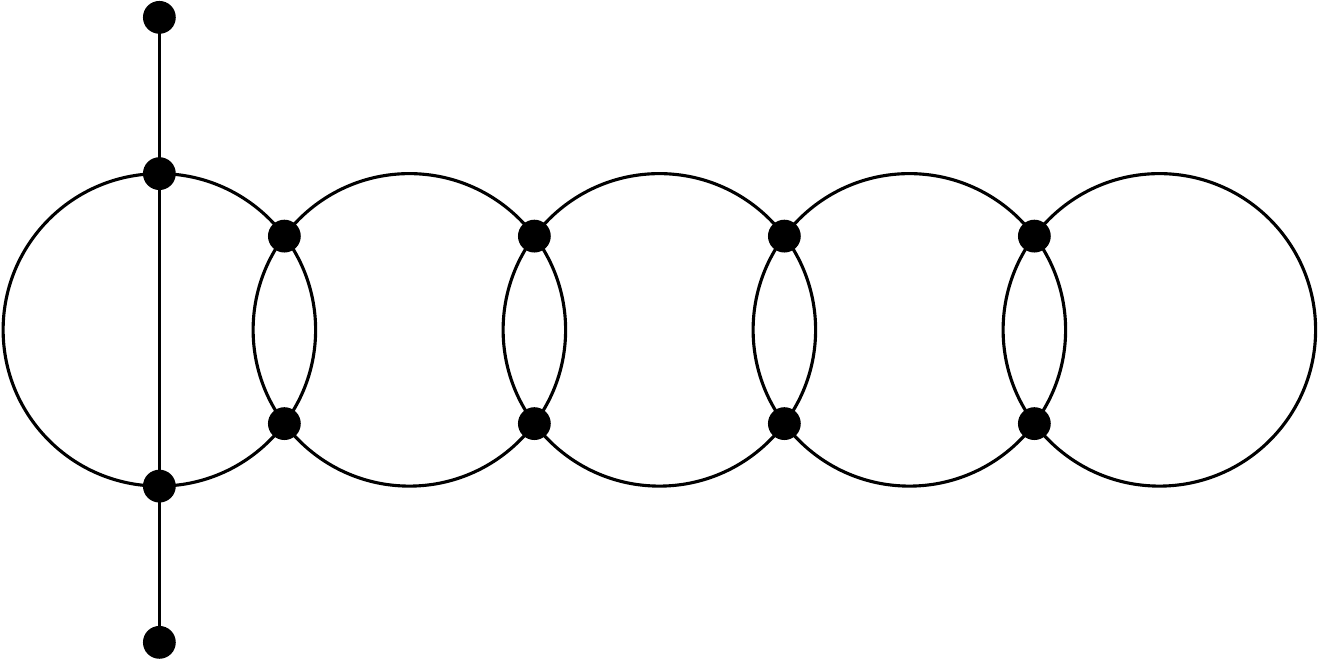}
\end{center}

\smallskip
\subsection{Melonic vacuum bubbles}

We will also consider the `vacuum' flavor of these constructions, in which the
two external vertices are identified; for example
\begin{center}
\includegraphics[scale=.4]{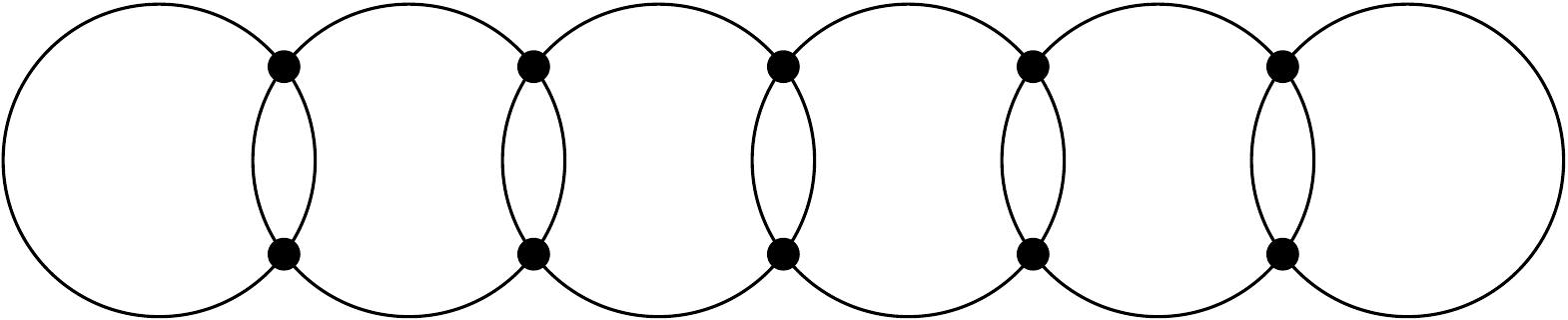}
\end{center}

\begin{defin}\label{def:vmg}
A {\em vacuum melonic graph\/} is a melonic graph without valence-$1$ vertices.
\qede\end{defin}

Vacuum melonic graph in which every vertex has valence~$4$ may be obtained by
iteratively applying the basic bananification $((1,3,1), p_s,k_s)$ starting from a $4$-banana. 
For example, the string of circles depicted above is produced by the construction
\[
(((4),0,1), ((1,3,1),1,1), ((1,3,1),2,2), ((1,3,1),3,2), ((1,3,1),4,2))\quad,
\]
while the construction
\[
(((4),0,1), ((1,3,1),1,1), ((1,3,1),1,1), ((1,3,1),1,1), ((1,3,1),1,1))
\]
yields the vacuum melonic graph
\begin{center}
\includegraphics[scale=.4]{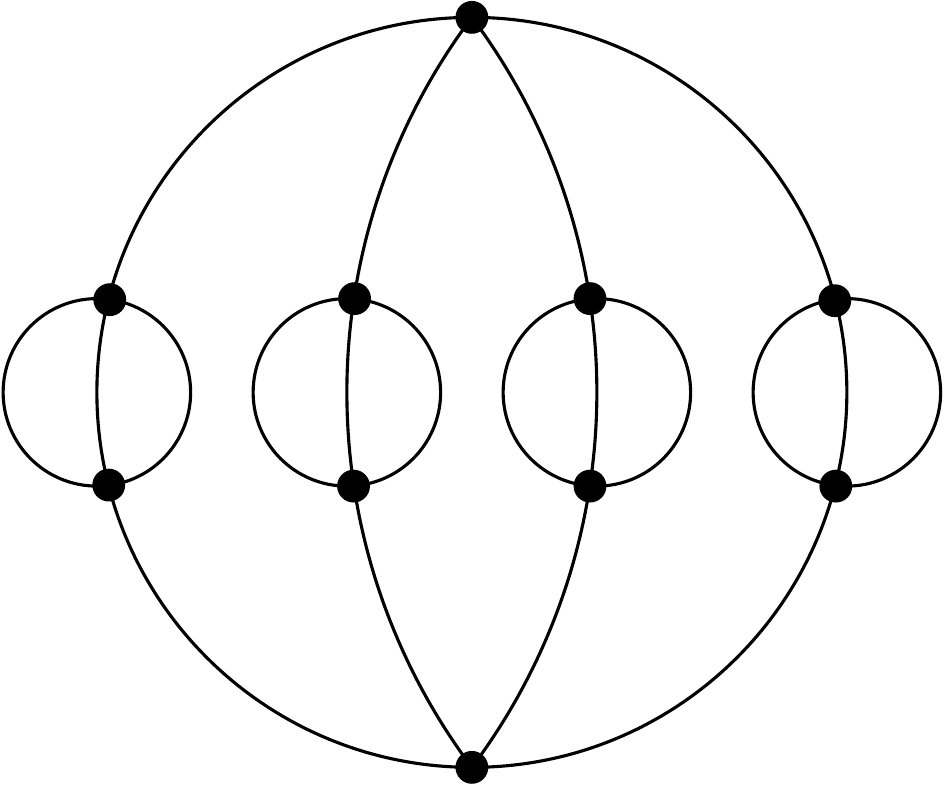}
\end{center}

Note that all vacuum valence-$4$ melonic graphs may also be constructed by starting
from a $2$-banana, performing iteratively the basic $(1,3,1)$ bananifications, and then
removing the two extra valence-$2$ vertices produced at the beginning. Indeed, the
$4$-banana itself admits such a description: the melonic construction
\[
(((2),0,1),((1,3,1),1,1))
\]
produces the $4$-banana graph wth two extra valence-$2$ vertices on one of the edges.
This alternative will be convenient in our computations concerning certain families of
vacuum melonic graphs in~\S\ref{Mav}.

\section{Grothendieck classes of melonic graphs}\label{MGc}

\subsection{The Grothendieck ring of varieties}
For $\cV_\K$ the category of varieties over a field $\K$ (which
we can here assume to be $\K=\Q$), the Grothendieck group
of varieties $K_0(\cV_\K)$ is the abelian group generated by isomorphism
classes $[X]$ of varieties $X\in \cV_\K$ with the inclusion-exclusion relation
\[
[X] =[Y]+ [X\smallsetminus Y]
\]
for closed subvarieties $Y\subset X$. This group may be given a ring structure
by defining $[X]\cdot [Y]=[X\times Y]$ and extending by linearity. 
Grothendieck classes, sometimes referred
to as virtual motives, behave like a universal Euler characteristic for algebraic
varieties. Grothendieck classes usually provide more computable information
about the nature of the motive of a variety. In particular, a Grothendieck class
is Tate if it is contained in the subring generated by the Lefschetz motive
$\bL=[\A^1]$ (the class of an affine line), or equivalently in the ring generated
by $\Tbb:=\Lbb-1$. Since the formulas we will obtain will naturally be polynomials 
in this class, and we will also be interested in expressing them in terms of the
class $\Sbb:=\Lbb-2$, we highlight their definitions.

\begin{defin}\label{def:Tbb}
We will denote by $\Tbb$ the class of the `torus' in the Grothendieck ring of
varieties, i.e., $\Tbb=[\Abb^1\smallsetminus \Abb^0]=\Lbb-1\in K_0(\cV_\K)$.
We will also denote by $\Sbb$ the class of the complement of three points
in $\Pbb^1$: $\Sbb=[\Pbb^1\smallsetminus \{0,1,\infty\}]=\Tbb-1$.
\qede\end{defin}

Varieties whose motive is in the category of
mixed Tate motives will have a Tate Grothendieck class. The converse 
holds conditionally (see \cite{Andre} for a discussion of this point). 

\smallskip
\subsection{Kirchhoff--Symanzik polynomials}

We consider the Kirchhoff--Symanzik polynomial of a graph $G$ with $n$ edges
\begin{equation}\label{KSpoly}
\Psi_G(t)=\sum_{T\subset G} \prod_{e\notin E(T)} t_e, 
\end{equation}
as a polynomial in variables $t=(t_1,\ldots, t_n)$ associated to the
edges of $G$, with the sum taken over all the spanning trees of the
graph. This is a homogeneous polynomial of degree $\ell=b_1(G)$,
the number of loops of $G$. Thus, we can consider the associated
projective graph hypersurface
\begin{equation}\label{XG}
X_G = \{ t=(t_1:\cdots : t_n) \in \P^{n-1}\,|\, \Psi_G(t)=0 \}.
\end{equation}

\smallskip

Up to renormalization of divergences, the Feynman parameter form of the
Feynman integral for the graph $G$, for a massless scalar field theory, is
of the form
\begin{equation}\label{UGp}
 U(G,p)=\frac{\Gamma(n-D\ell/2)}{(4\pi)^{\ell D/2}} \int_{\Delta_n} \frac{V_G(t,p)^{-n+\ell D/2}}{\Psi_G(t)^{D/2}} 
dt_1\cdots dt_n 
\end{equation}
as a function of the external momenta $p$, with $V_G(t,p)$ the second Symanzik polynomial (defined
in terms of cut sets of $G$), 
$D$ the spacetime dimension, and the integration performed on the $n$-simplex. In particular
(modulo divergences) the Feynman integral \eqref{UGp} can be regarded as the integration of an
algebraic differential form on a locus defined by algebraic equations (that is, a period) on the
complement of the hypersurface $X_G$, hence the interest in investigating the nature of the
motive of $\P^{n-1}\smallsetminus X_G$ through the computation of its Grothendieck class.
For a general introductory survey of parametric Feynman integrals and their relations to periods and motives
of graph hypersurfaces see \cite{Mar-book}. 

\smallskip

In the following we will consider both graphs with external edges and graphs (vacuum bubbles)
with no external edges. From the point of view of the parametric Feynman integral, the contribution
of the external edges with their assigned external momenta is encoded only in the 
second Symanzik polynomial $V_G(t,p)$, while the variables $t=(t_e)$ run over internal edges.
Thus, as long as the exponent satisfies $\ell D/2 \geq n$, with $\ell$ the number of loops, $n$ the
number of (internal) edges, and $D$ the spacetime dimension, the Feynman integral is computed on
the complement of the graph hypersurface defined by Kirchhoff--Symanzik 
polynomial $\Psi_G(t)$ that only depends on the internal edges of $G$. The 
Grothendieck class of the affine complement of the hypersurface of a graph $G$ (including
external edges) and of the same graph with the external edges removed are simply 
related by a product by a power of $\bL$ (the class of the affine line), hence it is
equivalent to compute one or the other. For the purpose of computing Grothendieck
classes, considering all graphs (both vacuum bubbles and
non-vacuum graphs) for a massless scalar theory with a self-interaction term of order $N$, so that the
corresponding Feynman graphs have (internal) vertices of valence $N$, is equivalent to
considering all vacuum bubble graphs for a massless scalar theory with self-interaction terms
of orders $v\leq N$. We will work for convenience with graphs with the external edges included.  

\smallskip

Up to the issue of renormalization, the Feynman integral \eqref{UGp} can then
be seen as a period of the graph hypersurface complement. The nature of
the motive of the graph hypersurface complement (detected by its Grothendieck class)
then provides information on the kind of numbers that can be obtained as periods.
The regularization and renormalization of the integral  \eqref{UGp} can also be
dealt with geometrically in terms of blowups or deformations. We will not discuss this
in the present paper and we refer the reader to \cite{Mar-book} for an overview and
to the references therein for more information.

\smallskip
\subsection{Grothendieck classes of graph hypersurface complements}

In previous work, especially~\cite{AluMa11b} and~\cite{AluMa11}, we have focused on the
essentially equivalent information given by the complement of the {\em affine cone\/}
$\hat X_G$ in its ambient affine space, and studied its class in the
Grothendieck group of varieties (the `motivic Feynman rule' of~\cite{AluMa11b}).
For short, we will refer to this class as the {\em Grothendieck class\/} of the graph or of
the corresponding melonic construction.

\begin{defin}\label{def:Grocla}
The {\em Grothendieck class\/} of $G$ (or of any of its melonic constructions) is
the class $\Ubb(G)=[\Abb^n\smallsetminus \hat X_G]\in K_0(\cV_\K)$ of the complement
of $\hat X_G$ in its natural ambient affine space $\A^n$, with $n$ the number of edges 
of $G$.
\qede\end{defin}

By construction, $\Ubb(G)$ is the class of a variety of dimension equal to the number
of edges of $G$.

In this section we will use the melonic constructions introduced in~\S\ref{BtmNc} to 
obtain a recursive computation of the Grothendieck class of a melonic graph. The class only 
depends on the isomorphism class of the resulting graph, so equivalent constructions
produce the same Grothendieck class.

We recall the following properties of $\bU(G)$.

\begin{itemized}
\item This invariant is `multiplicative', in the sense that 
\[
\Ubb(\Gamma_1\cup \Gamma_2) = \Ubb(\Gamma_1)\cdot \Ubb(\Gamma_2)
\]
if $\Gamma_1,\Gamma_2$ are graphs joined at one vertex (or disjoint);
\item For $\Gamma=$a loop, $\Ubb(\Gamma)=\Tbb$ (with $\Tbb$ as in 
Definition~\ref{def:Tbb});
\item For $\Gamma=$a single edge, $\Ubb(\Gamma)=\Lbb=\Tbb+1$;
\item If $\Gamma'$ is obtained from $\Gamma$ by splitting an edge, then 
$\Ubb(\Gamma')=(\Tbb+1)\cdot\Ubb(\Gamma)$;
\item If $\Gamma$ is an $m$-banana, $m>0$, then
\begin{equation}\label{eq:ban}
\Ubb(\Gamma)=\Bb_m:=m\Tbb^{m-1}+\Tbb\cdot\frac{\Tbb^m-(-1)^m}{\Tbb+1}
\end{equation}
(\cite{AluMa09} and \cite[Corollary~5.6]{AluMa11});
\item More generally: if $e$ is not a bridge or a looping edge, then for suitable 
polynomials $f_m$, $g_m$, $h_m$ in $\Tbb$,
\begin{equation}\label{eq:muledges}
\Ubb(\Gamma_{me}) = f_m \Ubb(\Gamma) + g_m \Ubb(\Gamma/e) 
+ h_m \Ubb(\Gamma\smallsetminus e)\quad,
\end{equation}
where:
\begin{itemized}
\item $\Gamma_{me}$ stands for the graph obtained from $\Gamma$ by replacing $e$ with $m$ 
parallel edges joining the same vertices as $e$;
\item $\Gamma\smallsetminus e = \Gamma_{0e}$ is $\Gamma$ with $e$ deleted; and
\item $\Gamma/e$ is $\Gamma$ with $e$ contracted.
\end{itemized}
This is a weak form of a deletion-contraction relation. Inductively, the coefficients
$f_m,g_m,h_m$ are determined by their value for $m\le 2$; in fact, we have
\begin{equation}\label{eq:coefs}
f_m=\frac{\Tbb^m-(-1)^m}{\Tbb+1},\quad
g_m=m\Tbb^{m-1}-\frac{\Tbb^m-(-1)^m}{\Tbb+1},\quad
h_m=\frac{\Tbb^m+(-1)^m \Tbb}{\Tbb+1}
\end{equation}
as obtained in~\cite[Corollary~5.7]{AluMa11}. Formulas for bridges and looping edges
are easier, as they follow immediately from the multiplicativity property.
\end{itemized}

\smallskip
\subsection{Recursion formulas for the Grothendieck classes}

The properties listed above, and particularly identity~\eqref{eq:muledges}, lead to 
recursion formulas for the computation of the Grothendieck class of the melonic graph
associated with a melonic construction, or equivalently of the corresponding tree.
(We will use the two descriptions interchangeably.)
We will abuse language and use both $\Ubb(T)$ and $\Ubb(G)$ for the Grothendieck
class of the melonic graph $G$ resulting from a melonic construction $T$.

The recursion formulas are based on the following observations.

Let $G$ be a melonic graph given by a melonic construction (tree) $T$.
\begin{itemized}
\item 
By Lemma~\ref{lem:redrep}, we may assume that the melonic construction is
reduced, i.e., nodes labeled $1$ are leaves of the tree $T$.
\item
If $T$ has depth $1$, i.e., the corresponding melonic construction consists of a
single tuple $((a_1,\dots, a_r), 0, 1)$, then 
\[
\Ubb(T) = \prod_{i=1}^r \Bb_{a_i} = \prod_{i=1}^r 
\left(a_i \Tbb^{a_i -1}+\Tbb\frac{\Tbb^{a_i}-(-1)^{a_i}}{\Tbb+1}\right)\quad.
\]
\end{itemized}
Indeed, the graph $G$ consists of a string of bananas in this case.

If $T=(t_1,\dots, t_n)$ has higher depth, consider the last state $t_n$. By
construction, the black nodes of $t_n$ are all leaves of $T$. 
\begin{itemized}
\item If $t_n=((a),p,k)$, then an equivalent 
tree of depth $n-1$ is obtained by omitting~$t_n$ and increasing the label
of the $k$-th leaf of $t_p$ by $a-1$.
\begin{center}
\includegraphics[scale=.5]{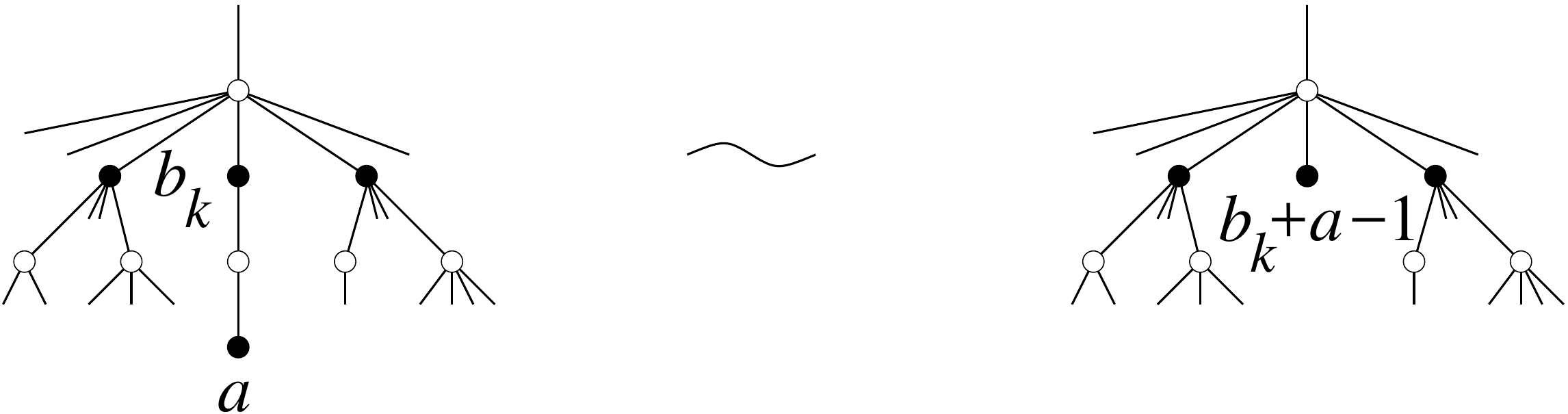}
\end{center}
Indeed, this step of the construction simply replaces $1$ edge in the $k$-th
banana of $t_p$ by $a$ parallel edges.
\item If $t_n=((\underbrace{1,\dots, 1}_{\text{$r$ times}}),p,k)$,
then let $T'=(t_1,\dots, t_{n-1})$ be the construction obtained by omitting the last
stage. Then $\Ubb(T)=(\Tbb+1)^{r-1} \Ubb(T')$. Indeed, the effect of $t_n$ is to
split one edge in the $k$-th banana of $t_p$ a total of $r-1$ times.
\begin{center}
\includegraphics[scale=.4]{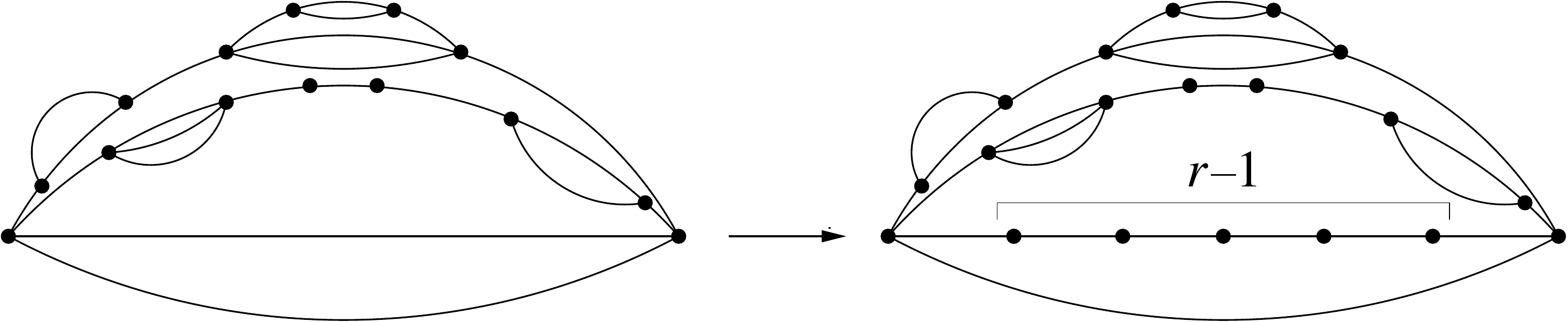}
\end{center}
\item
We may therefore assume that $t_n=((a_1,\dots, a_r), p,k)$ with $r>1$ and such
that $a_m=\max(a_1,\dots, a_r)>1$. The effect of $t_n$ is to replace one edge in
the $k$-th banana of $t_p$ by a string of $a_1,\dots, a_m, \dots, a_r$-bananas;
this is the same as replacing that edge by a string of 
$a_1,\dots, a_{m-1},1,a_{m+1}, \dots, a_r$-bananas, and then replacing the 
resulting single edge $e$ by $a_m$ parallel edges.

\begin{center}
\includegraphics[scale=.4]{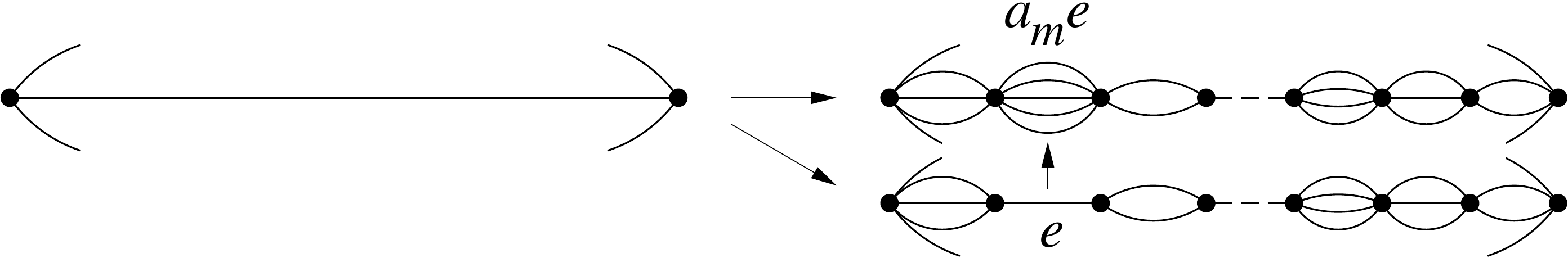}
\end{center}

Let $G'$ be the graph obtained from $G$ by replacing the $a_m$-banana by the 
single edge $e$. With notation as in~\eqref{eq:muledges}, we have $G=G'_{a_me}$.

\begin{claim}
The edge $e$ is not a bridge (or a looping edge) of $G'$.
\end{claim}

\begin{proof}
This follows from the assumption that $T$ be reduced. Indeed, as a consequence
the $k$-th banana of $t_p$ does not consist of a single edge; hence removing one
edge of this banana does not disconnect the graph. Since the edge $e$ is one
edge in a subdivision of one edge of the $k$-th banana of $t_p$, removing it does
not disconnect the graph. (And the construction never produces looping edges,
therefore $e$ is not a looping edge.)
\end{proof}

It follows that we can use~\eqref{eq:muledges} to relate $\Ubb(G)$ to the Grothendieck
classes of $G'$ and associated graphs:
\begin{equation}\label{eq:recurG}
\Ubb(G) = f_{a_m} \Ubb(G') + g_{a_m} \Ubb(G'/e) 
+ h_{a_m} \Ubb(G'\smallsetminus e)
\end{equation}
with $f_{a_m},g_{a_m},h_{a_m}$ as in~\eqref{eq:coefs}.

Now:

--- $G'$ is a melonic graph: its construction $T'$ is obtained from $T$ by replacing
\begin{align*}
t_n&=((a_1,\dots, a_{m-1},a_m,a_{m+1},\dots, a_r),p,k)
\intertext{by}
t'_n&=((a_1,\dots, a_{m-1},1,a_{m+1},\dots, a_r),p,k)\quad.
\end{align*}
Pictorially:
\begin{center}
\includegraphics[scale=.5]{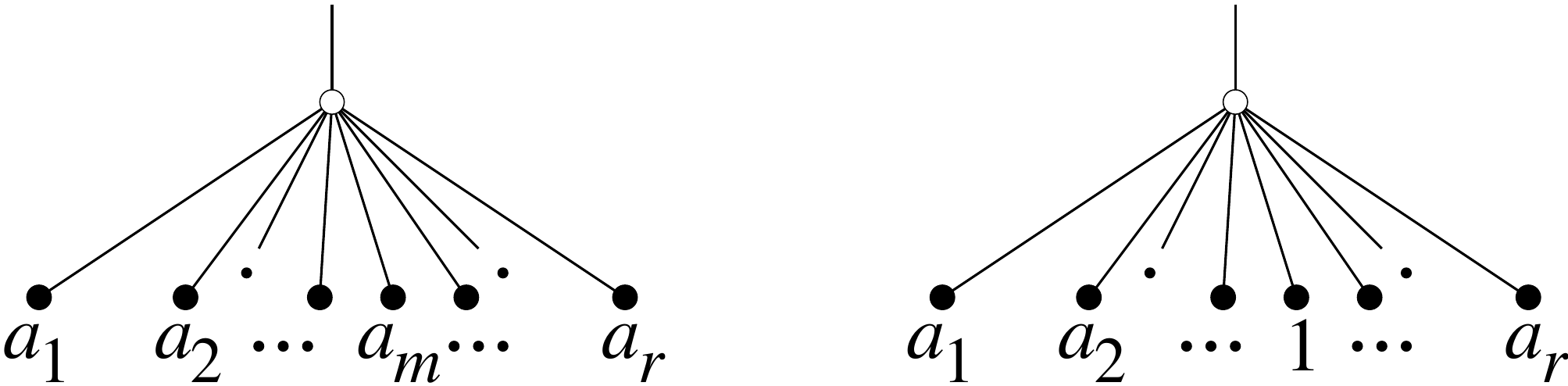}
\end{center}

---The contraction $G'/e$ is also a melonic graph: its construction $T''$ is obtained 
from~$T$ by omitting $a_m$ in $t_n$, i.e., replacing 
\[
t_n=((a_1,\dots, a_{m-1},a_m,a_{m+1},\dots, a_r),p,k)
\]
by
\[
t_n''=((a_1,\dots, a_{m-1},a_{m+1},\dots, a_r),p,k)\quad.
\]
Since $r>1$, the tuple $(a_1,\dots, a_{m-1},a_{m+1},\dots, a_r)$ is non-empty,
as needed (cf.~Definition~\ref{def:melcon}).
Pictorially:
\begin{center}
\includegraphics[scale=.5]{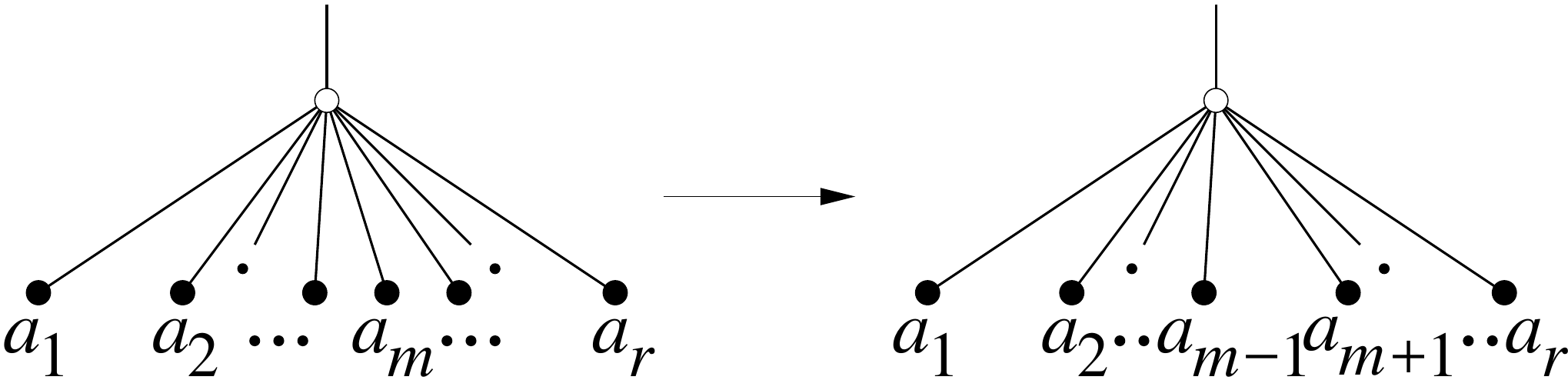}
\end{center}

---The deletion $G'\smallsetminus e$ is {\em not\/} a melonic graph; it is obtained by
replacing one edge of the $k$-th banana of $t_p$ by two disconnected strings of
bananas attached at the vertices of that edge:
\begin{center}
\includegraphics[scale=.5]{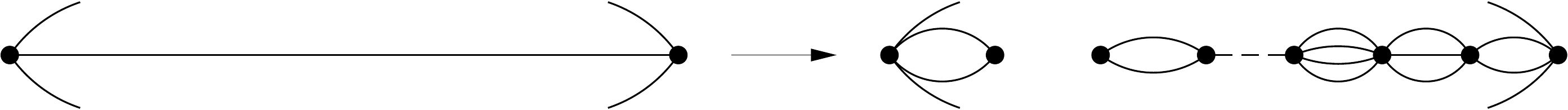}
\end{center}
Let $T'''$ be the list obtained from $T$ by omitting $t_n$ and decreasing
by $1$ the order~$b_k$ of the $k$-th banana in $t_p$. Since $T$ is assumed to be reduced,
$b_k>1$; therefore, $T'''$ is still a melonic construction. (Note that, however, $T'''$ may be
{\em non-reduced.\/} This is the reason forcing us to consider non-reduced melonic 
constructions.) The graph $G'\smallsetminus e$
is obtained from the melonic graph corresponding to $T'''$ by attaching two strings of bananas
to two vertices, and it follows that
\[
\Ubb(G'\smallsetminus e) = \left(\prod_{i=1}^{m-1} \Bb_{a_i}\right)
\left(\prod_{i=m+1}^{r} \Bb_{a_i}\right) \Ubb(T''')\quad.
\]
In conclusion, \eqref{eq:recurG} may be rewritten
\[
\Ubb(T)=f_{a_m}\Ubb(T')+g_{a_m}\Ubb(T'') + \left(\prod_{i=1}^{m-1} \Bb_{a_i}\right)
\left(\prod_{i=m+1}^{r} \Bb_{a_i}\right) h_{a_m}\Ubb(T''')\quad,
\]
or, more explicitly:

\begin{prop}
With notation as above,
\begin{align*}
\Ubb(T) 
&=\frac{\Tbb^{a_m}-(-1)^{a_m}}{\Tbb+1} \Ubb(T') \\
&+\left(a_m\Tbb^{a_m-1}-\frac{\Tbb^{a_m}-(-1)^{a_m}}{\Tbb+1}\right)\Ubb(T'') \\
&+\left(\prod_{i=1}^{m-1} \Bb_{a_i}\right)
\left(\prod_{i=m+1}^{r} \Bb_{a_i}\right) 
\frac{\Tbb^{a_m}+(-1)^{a_m} \Tbb}{\Tbb+1}\Ubb(T''')\quad.
\end{align*}
\end{prop}
Since $T'$, $T''$, $T'''$ all correspond to melonic graphs with fewer edges than $G$,
the corresponding Grothendieck classes are recursively known, and determine $\Ubb(G)
=\Ubb(T)$.
\end{itemized}

\begin{corol}\label{cor:Tate}
The graph hypersurface of a melonic graph $G$ determines a mixed Tate motive;
the Grothendieck class $\Ubb(G)$ is a polynomial in $\Lbb$ of degree equal to the
number of edges of $G$.
\end{corol}

\begin{proof}
The recursion implies immediately that $\Ubb(G)$ is a polynomial in $\Tbb$,
therefore in $\Lbb=\Tbb+1$. 
By construction, $\Ubb(G)$ is the class in the Grothendieck group of a variety
of dimension equal to the number of edges of $G$, so the statement follows.
\end{proof}

\smallskip
\subsection{Positivity and log-concavity}

The class of a melonic graph can of course also be written as a polynomial in 
the class $\Sbb=\Tbb-1=[\Pbb^1\smallsetminus \{0,1,\infty\}]$. Remarkably, these 
polynomials are `positive', in the following sense.

\begin{corol}
Let $G$ be a melonic graph. Then $\Ubb(G)=P(\Sbb)$ for a polynomial
$P(t)=a_n t^n+\cdots + a_1 t + a_0\in \Zbb[t]$ with nonnegative integer coefficients.
\end{corol}

\begin{proof}
Given the recursion, it suffices to observe that the classes of banana graphs,
$\Bb_m(\Tbb)=\Bb_m(\Sbb+1)$, and the coefficients $f_m,g_m,h_m$ are all
positive as polynomials in~$\Sbb$. The key observation is the following.

\begin{claim}\label{claim:pos}
The class 
\[
\frac{\Tbb^m-(-1)^m}{\Tbb+1}=\frac{(\Sbb+1)^m-(-1)^m}{\Sbb+2}
\]
is positive in $\Sbb$; in fact,
\[
\frac{\Tbb^m-(-1)^m}{\Tbb+1} = \sum_{j=1}^{m-1} \sum_{i=1}^{\frac m2}
\binom{m-2i}{j-1} \,\Sbb^j+\left\{\begin{aligned} 0&\quad \text{if $m$ is even} \\
1&\quad \text{if $m$ is odd}\end{aligned}\right.\quad.
\]
\end{claim}

This is a straightforward computation, left to the reader. Given Claim~\ref{claim:pos},
it follows immediately that
\[
\Bb_m=m\Tbb^{m-1}+\Tbb\frac{\Tbb^m-(-1)^m}{\Tbb+1}\quad,\quad
f_m=\frac{\Tbb^m-(-1)^m}{\Tbb+1}\quad,\quad
h_m=\frac{\Tbb^m+(-1)^m \Tbb}{\Tbb+1}
\]
are positive in $\Sbb$. As for 
\[
g_m=m\Tbb^{m-1}-\frac{\Tbb^m-(-1)^m}{\Tbb+1}\quad,
\]
the required positivity follows from the fact that for all $m,i\ge 1,j$
\[
\binom {m-1}j\ge \binom{m-2i}{j-1}
\]
which is clear, as $\binom {m-1}j= \binom{m-2}{j-1}+\binom{m-2}j\ge \binom{m-2}{j-1}
\ge \binom{m-2i}{j-1}$ for $i\ge 1$.
\end{proof}

\begin{example}
The melon-tadpole graph of Figure~\ref{meltadFig} consists of a $4$-banana
tadpole, with class $\Bb_4=(\Tbb+1)(\Tbb^2+2\Tbb-1) \Tbb$, and of a melonic
part which may be constructed by
\[
\big( ((4),0,1),((1,3,1),1,1) \big)
\]
i.e., by the labeled tree
\begin{center}
\includegraphics[scale=.4]{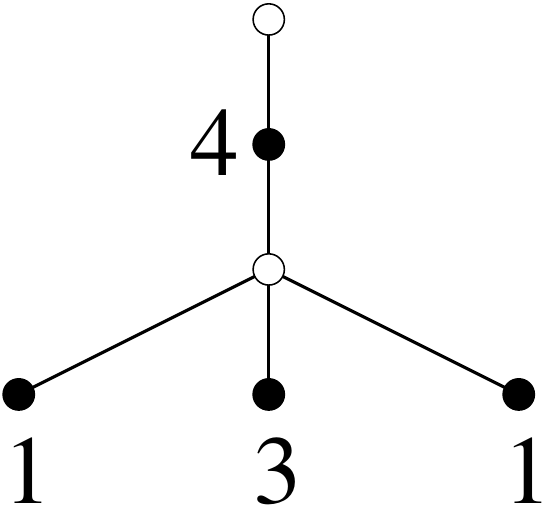}
\end{center}
The recursion obtained above computes the Grothendieck class of this melonic
graph to be
\[
\Tbb^2(\Tbb+1)^4(\Tbb^2+3\Tbb-2)\quad.
\]
The conclusion is that the Grothendieck class for the graph in Figure~\ref{meltadFig}
equals
\[
\Tbb^3(\Tbb+1)^7(\Tbb^2+3\Tbb-2)(\Tbb^2+2\Tbb-1)
=(\Sbb+1)^3(\Sbb+2)^5(\Sbb^2+4\Sbb+2)(\Sbb^2+5\Sbb+2)\quad:
\]
indeed, the graph may be obtained by splitting one edge in each of the two components
(which has the effect of multiplying each Grothendieck class by $\Tbb+1$), and then
joining the resulting graphs at the newly created vertices, i.e., multiplying together the 
two resulting Grothendieck classes.
\qede\end{example}

\smallskip

Positivity as a polynomial in the class $\Tbb$ is a torification of the Grothendieck class,
which may or may not be induced by a geometric torification of the underlying variety,
see \cite{ManMar}. The presence of a torified Grothendieck class has consequences
in terms of ``geometry over the field with one element", \cite{BejMar}, \cite{ManMar}.
One can similarly ask whether the positivity of the Grothendieck class as a function 
of $\Sbb$ is induced by an underlying geometric structure and whether such a 
structure carries arithmetic significance. For example, the Grothendieck class of
the moduli spaces $\mathcal M_{0,n}$ of genus zero curves with marked points have the
simple expression 
$$ [\cM_{0,n}] = \binom{\Sbb}{n-3} (n-3)! $$
in terms of the class $\Sbb$, with $[\mathcal M_{0,4}]=\Sbb$. 
However, these classes are not positive in $\Sbb$, while
the classes of the $\overline{ \mathcal M}_{0,n}$ moduli spaces satisfy positivity (both in $\Sbb$
and in $\Tbb$), see \cite{ManMar}.

\smallskip

Another feature of the polynomials expressing the classes in terms of $\Sbb$ appears
to be the following.

\begin{conj}
Let $G$ be a melonic graph, and let $\Ubb(G)=a_0+a_1 \Sbb+\cdots + a_n \Sbb^n$ be 
its Grothendieck class. Then the sequence $a_0,a_1,\dots, a_n$ is log concave, i.e.,
$a_{i-1} a_{i+1}\le a_i^2$ for $0<i<n$.
\end{conj}

We have verified this conjecture for all melonic graphs with $\le 13$ edges
and for hundreds of individual examples from the families of melonic
graphs considered in this paper.
\begin{example}
As polynomials in $\Sbb$, the Grothendieck classes of all possible melonic graphs
with $7$ edges are
\begin{align*}
&\left( \Sbb+1 \right) ^{3} \left( \Sbb+2\right) ^{4} \\ 
&\left( \Sbb+1 \right) ^{2}\left( \Sbb+2 \right) ^{5} \\ 
&\left( \Sbb+1 \right) \left( \Sbb+2 \right) ^{6}   \\ 
&\left( \Sbb+2 \right) ^{7} \\ 
&\left( \Sbb+1 \right) ^{3} \left( \Sbb+2 \right) ^{3}\left( \Sbb+3 \right)   \\
&\left( \Sbb+1 \right) ^{2} \left( \Sbb+2\right) ^{4} \left( \Sbb+3 \right)  \\ 
&\left( \Sbb+1 \right) ^{3}  \left( \Sbb+2 \right) ^{3} \left( \Sbb+4 \right) \\
&\left( \Sbb+1 \right) ^{4} \left( \Sbb+2 \right) ^{2} \left( \Sbb+5 \right)  \\
&\left( \Sbb+1 \right) ^{2} \left( \Sbb+2 \right) ^{3} \left( {\Sbb}^{2}+4\,\Sbb+2 \right)  \\ 
&\left( \Sbb+1 \right) \left( \Sbb+2\right) ^{4} \left( {\Sbb}^{2}+4\,\Sbb+2 \right)  \\
&\left( \Sbb+1 \right) ^{2} \left( \Sbb+2 \right) ^{3} \left( {\Sbb}^{2}+5\,\Sbb+2 \right)  \\ 
&\left( \Sbb+1 \right) ^{2} \left( \Sbb+2 \right) ^{3}  \left( {\Sbb}^{2}+5\,\Sbb+5 \right)   \\ 
&\left( \Sbb+1 \right) ^{2} \left( \Sbb+2 \right) ^{2} \left( {\Sbb}^{3}+6\,{\Sbb}^{2}
+7\,\Sbb+3 \right)  \\ 
&\left( \Sbb+1 \right) \left( \Sbb+2 \right) ^{3} \left( {\Sbb}^{3}+6\,{\Sbb}^{2}
+7\,\Sbb+3 \right)   \\ 
&\left( \Sbb+1 \right) \left( \Sbb+2 \right) ^{3} \left( {\Sbb}^{3}+6\,{\Sbb}^{2}
+9\,\Sbb+3 \right) \\
&\left( \Sbb+1 \right) \left( \Sbb+2\right) ^{2} \left( {\Sbb}^{4}+8\,{\Sbb}^{3}
+15\,{\Sbb}^{2}+12\,\Sbb+3 \right)  \\ 
&\left( \Sbb+1 \right) \left( \Sbb+2 \right) ^{2}\left( {\Sbb}^{4}+8\,{\Sbb}^{3}
+19\,{\Sbb}^{2}+16\,\Sbb+5 \right)    \\
&\left( \Sbb+1 \right)  \left( \Sbb+2\right)  \left( {\Sbb}^{5}+10\,{\Sbb}^{4}+26\,{\Sbb}^{3}
+31\,{\Sbb}^{2}+17\,\Sbb+4\right) 
\end{align*}
One may verify that all these polynomials are log-concave (in the sense that the
coefficients of their expansions are log-concave sequences). 
The number of distinct Grothendieck classes for melonic graphs with $n$ edges
is
\[
1, 2, 2, 4, 6, 11, 18, 33, 59, 114,220,454,954\dots
\] 
respectively as $n=1,2,3,\dots$
\qede\end{example}

\smallskip

The log-concavity property of the Grothendieck classes implies similar properties
for the image of these classes under any motivic measure, meaning a ring
homomorphism $\mu: K_0(\cV)\to R$. Such measures include the topological
Euler characteristic and the Hodge--Deligne polynomials (for complex varieties)
or the counting of points (for varieties over finite fields). As discussed in \cite{Huh},
the presence of a log-concave structure is usually a sign of the presence of an
underlying richer kind of structure, in the form of Hodge-de Rham relations.
These can be seen as a broad combinatorial generalization of the
setting of the Grothendieck standard conjectures for algebraic cycles.
Such combinatorial Hodge-de Rham relations arise, for example, in the
context of the log-concavity property of characteristic polynomial of
matroids, \cite{AdiHuhKa}. Thus, the observed log-concavity of the Grothendieck
classes of the graph hypersurface complements as a function
of the $\Sbb$ variable suggest the presence of a more interesting underlying
geometric structure in this Hodge-de Rham sense. 

\smallskip

While the Grothendieck classes are positive in the class~$\Sbb$ and display this
intriguing property, we will persist in
using~$\Tbb$ in most of the examples that follow, since the coefficients of the powers
of $\Tbb$ in these classes tend to be smaller.

\section{Explicit computations, I}\label{MEcacI}

The recursion obtained in~\S\ref{MGc} is easily implemented in any symbolic
manipulation package, and this makes it possible to explore the landscape of Grothendieck
classes for natural families of  melonic graphs. We will present a selection of such
formulas in the sections that follow. While we are able to prove these formulas (see~\S\ref{Mp}),
the recursion formulas were key to discovering them, and often the numerical evidence
we gathered was quite sufficient to convince us of their truth. 
It would be worthwhile studying other natural families of melonic graphs using the same method.

\smallskip

In this section we focus on melonic graph with internal vertices of valence~$4$, and we
will use the shorthand for such graphs introduced in~\S\ref{BtmNc}.

\begin{example}
For a simple valence-$4$ example that can be computed without employing the 
full recursion from~\S\ref{MGc}, we can consider the graph
\begin{center}
\includegraphics[scale=.5]{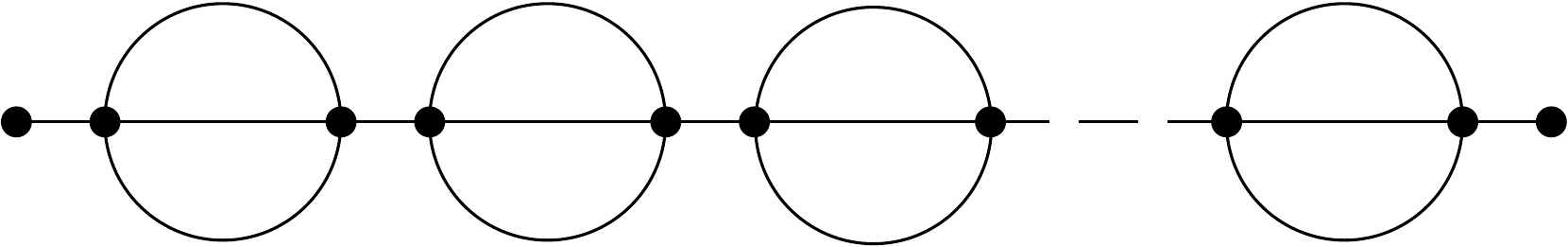}
\end{center}
with $n$ circles.
A corresponding melonic construction is $(0,1^-,2^-,3^-,\dots, (n-1)^-)$. This construction is 
non-reduced; a reduced alternative is simply the $1$-stage construction
\[
((1,3,1,3,1,3,\dots,1),0,1)\quad.
\]

The corresponding Grothendieck class is a product of classes of $3$-bananas 
and $(\Tbb+1)$-factors, accounting for the external and internal single edges.
Explicitly, the class equals
\[
\Bb_3^n\cdot (\Tbb+1)^{n+1} = \Tbb^n (\Tbb+1)^{3n+1}\quad.
\]
for $n$ circles.
\qede\end{example}

\begin{example}\label{ex:gamman}
At the opposite end of the spectrum, and more interestingly, consider the
valence-$4$ melonic graphs $\Gamma_n$ constructed by
$(0,1^+,2^+,3^+,\dots, (n-1)^+)$. These are graphs of the form
\begin{center}
\includegraphics[scale=.4]{nbubbles}
\end{center}
with $n$ circles.

\smallskip
\subsection{Recursion for the $\Gamma_n$ graphs}

The graph $\Gamma_n$ has $4n+1$ edges, so by Corollary~\ref{cor:Tate} its 
Grothendieck class is a polynomial in $\Tbb$ of degree $4n+1$. For $n=1,\dots, 7$ 
the recursion obtained in~\S\ref{MGc} yields the following Grothendieck classes:
\begin{align*}
n=1&:\qquad \Tbb^1(\Tbb+1)^3\cdot (\Tbb+1) \\
n=2&:\qquad \Tbb^2(\Tbb+1)^5\cdot (\Tbb^2+3\Tbb) \\
n=3&:\qquad \Tbb^3(\Tbb+1)^7\cdot (\Tbb^3+5\Tbb^2+4\Tbb-2) \\
n=4&:\qquad \Tbb^4(\Tbb+1)^9\cdot (\Tbb^4+7\Tbb^3+12\Tbb^2-4) \\
n=5&:\qquad \Tbb^5(\Tbb+1)^{11}\cdot (\Tbb^5+9\Tbb^4+24\Tbb^3+14\Tbb^2-12\Tbb-4) \\
n=6&:\qquad \Tbb^6(\Tbb+1)^{13}\cdot (\Tbb^6+11\Tbb^5+40\Tbb^4+48\Tbb^3-8\Tbb^2-28\Tbb) \\
n=7&:\qquad \Tbb^7(\Tbb+1)^{15}\cdot 
(\Tbb^7+13\Tbb^6+60\Tbb^5+110\Tbb^4+40\Tbb^3-72\Tbb^2-32\Tbb+8)
\end{align*}
Identifying the pattern underlying these expressions is an interesting challenge.
\begin{itemized}
\item 
Define polynomials $a_k(r,t)\in \Zbb[r,t]$ for $k\ge 0$ by the power series expansion
\begin{equation}\label{eq:cos1}
e^{rs} \cos((r^2-rt)^{\frac 12}s)=\sum_{k\ge 0} a_k(r,t)\frac{s^k}{k!}\quad; 
\end{equation}
\item
In turn, define polynomials $A_n(t)\in \Zbb[t]$ for $n\ge 0$ by the equality of formal 
power series
\[
\sum_{k\ge 0} a_k(r,t) = \sum_{n\ge 0} A_n(t) r^n\quad.
\]
(Since $t$ only appears in the product $rt$ in~\eqref{eq:cos1}, it is clear that 
$A_n(t)$ is indeed a polynomial, of degree at most $n$. In fact, $\deg A_n=n$.)

\begin{prop}\label{prop:GCgamman}
With $\Gamma_n$ as above, 
$\Ubb(\Gamma_n)=\Tbb^n (\Tbb+1)^{2n+1}\cdot A_n(\Tbb)$
for $n\ge 1$.
\end{prop}
\end{itemized}

Proposition~\ref{prop:GCgamman} may be easily verified for low values of $n$;
our computer implementation takes a few seconds to verify it for $n=1,\dots,100$.
We will prove Proposition~\ref{prop:GCgamman} in~\S\ref{Mp}.

The definition given above for the polynomials $A_n(t)$ is of course just one
choice among many. An alternative (and perhaps simpler) formulation will be 
given in~\S\ref{MEcacII}. The most explicit version of the same result is the following.

\begin{corol}
\begin{align*}
\Ubb(\Gamma_n)&=
\Tbb^n (\Tbb+1)^{2n+1}\cdot \sum_{0\le i\le j} \binom{n+i}{2j}\binom ji (-1)^{j-i} \Tbb^i \\
&=(\Sbb+1)^n (\Sbb+2)^{2n+1}\cdot\sum_{0\le i, j} \binom{n-j}{i}\binom {i+j-1}j 2^{n-i-j} \Sbb^i\quad.
\end{align*}
\end{corol}

The straightforward details are left to the reader.
\qede\end{example}

\begin{example}\label{ex:gammavn}
A similar pattern holds for {\em vacuum\/} graphs analogous to those
considered in Example~\ref{ex:gamman}. Let $\Gamma_n'$ denote the graph
\begin{center}
\includegraphics[scale=.4]{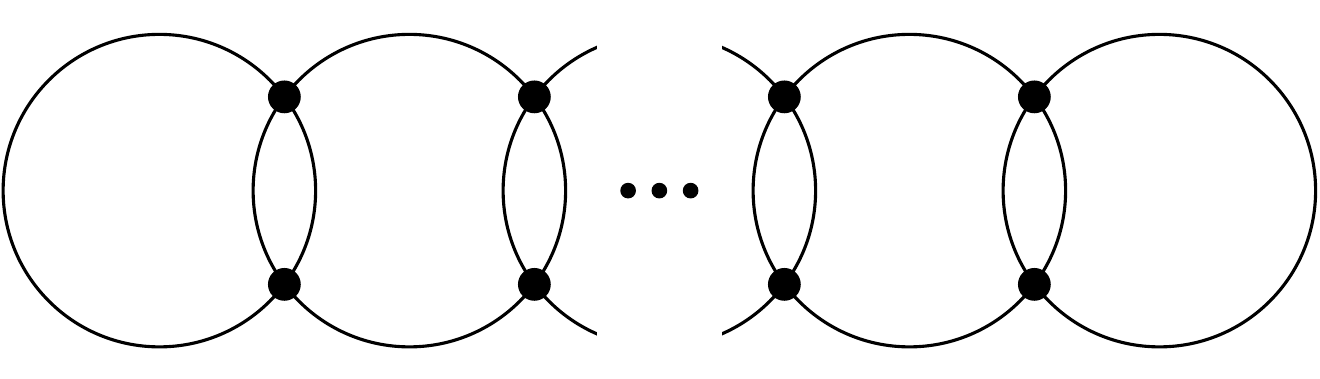}
\end{center}
with $n$ circles. As observed at the end of~\S\ref{BtmNc}, these graphs are
also melonic: their construction is $((4),0,1)$ for two circles and 
\[
((4),0,1),((1,3,1),1,1),((1,3,1),2,2),\dots, ((1,3,1),n-2,2)
\]
for a $n\ge 3$ circles.
For $n\ge 2$, Corollary~\ref{cor:Tate} implies that
$\Ubb(\Gamma'_n)$ is a polynomial of degree $4n-4$ in $\Tbb$.
Applying the recursion obtained in~\S\ref{MGc} we obtain the
following expressions for $\Ubb(\Gamma'_n)$ $n=2,\dots,7$:\begin{align*}
n=2&:\qquad \Tbb^1(\Tbb+1)^1\cdot (\Tbb^2+2\Tbb-1) \\
n=3&:\qquad \Tbb^2(\Tbb+1)^3\cdot (\Tbb^3+4\Tbb^2+\Tbb-2) \\
n=4&:\qquad \Tbb^3(\Tbb+1)^5\cdot (\Tbb^4+6\Tbb^3+7\Tbb^2-4\Tbb-2) \\
n=5&:\qquad \Tbb^4(\Tbb+1)^7\cdot (\Tbb^5+8\Tbb^4+17\Tbb^3+2\Tbb^2-12\Tbb) \\
n=6&:\qquad \Tbb^5(\Tbb+1)^9\cdot (\Tbb^6+10\Tbb^5+31\Tbb^4+24\Tbb^3-22\Tbb^2-16\Tbb+4) \\
n=7&:\qquad \Tbb^6(\Tbb+1)^{11}\cdot 
(\Tbb^7+12\Tbb^6+49\Tbb^5+70\Tbb^4-8\Tbb^3-64\Tbb^2-4\Tbb+8)
\end{align*}
\begin{itemized}
\item 
Define rational functions $a'_k(r,t)\in \Zbb[t](r)$ for $k\ge 0$ by the power series expansion
\begin{equation}\label{eq:cos2}
\cos\left(\frac{(r^2-rt)^{\frac 12}}{1-r}s\right)=\sum_{k\ge 0} a'_k(r,t)\frac{s^k}{k!}\quad; 
\end{equation}
that is, let $a'_k(r,t)=0$ for $k$ odd and $a'_{2\ell}(r,t)=\frac 1{(2\ell)!} \frac{r^\ell(t-r)^\ell}
{(1-r)^{2\ell}}$.
\item
Define polynomials $A'_n(t)\in \Zbb[t]$ for $n\ge 0$ by the equality of formal 
power series
\[
\sum_{k\ge 0} a'_k(r,t) = \sum_{n\ge 0} A'_n(t) r^n\quad.
\]
(Again, $A'_n(t)$ is clearly a polynomial, and $\deg A'_n=n$.)
\begin{prop}\label{prop:GCgammanv}
With $\Gamma'_n$ as above, 
$\Ubb(\Gamma'_n)=\Tbb^{n-1} (\Tbb+1)^{2n-3}\cdot A'_n(\Tbb)$
for $n\ge 2$.
\end{prop}
\end{itemized}

Again, Proposition~\ref{prop:GCgammanv} may be easily verified by computer,
using the recursion formula obtained in~\S\ref{MGc}, for (hundreds of) low 
values of $n$. Proposition~\ref{prop:GCgammanv} will also be proved in~\S\ref{Mp}.

\begin{corol}
\[
\Ubb(\Gamma'_n)=
\Tbb^{n-1}\Tbb^{2n-3}\sum_{0\le i\le j} \binom{n+i-1}{2j-1} \binom ji (-1)^{j-i}\Tbb^i\quad.
\]
\end{corol}

\smallskip
\subsection{Relations of vacuum and non-vacuum graphs}

A particularly careful reader may notice the following relation from the data shown
above:
\begin{equation}\label{eq:relvnv}
A'_n(t) = A_n(t)-A_{n-1}(t)\quad.
\end{equation}
This relation is not a coincidence; it follows from a general formula relating
{Gro\-then\-dieck} classes of melonic vacuum graphs to classes of related non-vacuum
graphs. We will prove this formula in~\S\ref{Mav}.

An even more careful reader may guess the divisibility relation
\begin{equation}\label{eq:divi}
\Ubb(\Gamma_n) \, |\, \Ubb(\Gamma'_{2n+1})\quad:
\end{equation}
for example, 
\[
A'_7(t) = (t^2+5t+2)(t^2+2t-2)\cdot A_3(t)\quad.
\]
The relation~\eqref{eq:divi} will also be obtained as a corollary of a more general
result on melonic vacuum graphs, Proposition~\ref{prop:stars} in~\S\ref{Mav}
(see Remark~\ref{rem:diviex}).
\qede\end{example}

More examples of computations of Grothendieck classes for melonic vacuum graphs
will be given in~\S\ref{Mav}.

The proofs we will discuss in~\S\ref{Mp} will clarify the presence of the factors $\Tbb^i
(\Tbb+1)^j$ in the Grothendieck classes for the valence-$4$ graphs considered in this
section. The example $(0,1^+,1^+,1^+)$,
\begin{center}
\includegraphics[scale=.5]{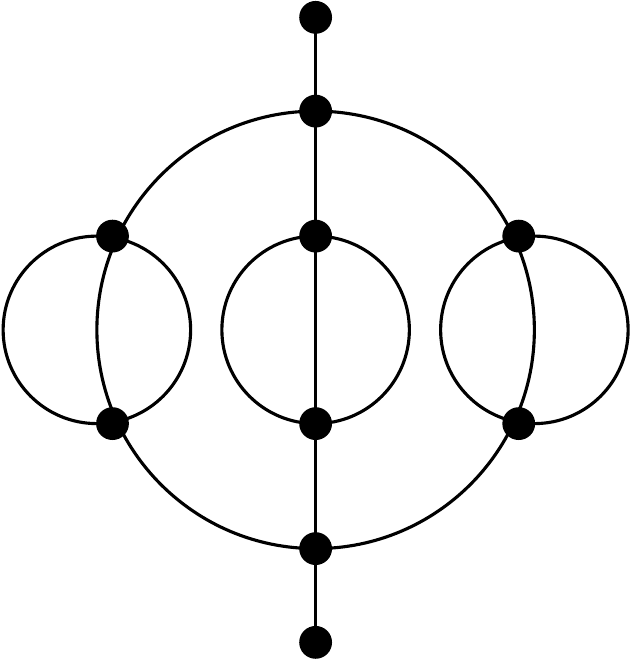}
\end{center}
with Grothendieck class
\[
\Tbb^3 (\Tbb+1)^{10} (\Tbb+3) (\Tbb^3+3\Tbb^2-3\Tbb+1)\quad,
\]
shows that $\Tbb^n (\Tbb+1)^{2n+1}$ is {\em not\/} a common factor of the Grothendieck
classes of all $n$-stage valence-$4$ melonic constructions.

\section{Explicit computations, II}\label{MEcacII}

\subsection{Rational generating functions for $\Gamma_n$ graphs}
While~\S\ref{MEcacI} and~\S\ref{Mav} focus on valence-$4$ graphs,
the same types of computation can be carried out for melonic graphs of any
fixed valence for internal vertices. In order to obtain simpler statements, 
it is helpful to express the results stated in Examples~\ref{ex:gamman} 
and~\ref{ex:gammavn} in terms of rational generating functions.

\begin{prop}\label{prop:both}
With notation as in Examples~\ref{ex:gamman} and~\ref{ex:gammavn},
and setting $A_0(t)=A'_0(t)=1$,
\begin{align*}
\sum_{n\ge 0} A_n(t) r^n &= \frac{1-r}{1-(2+t)r+2r^2}\quad; \\
\sum_{n\ge 0} A'_n(t) r^n &= \frac{(1-r)^2}{1-(2+t)r+2r^2}\quad.
\end{align*}
\end{prop}

\begin{proof}
We verify that the polynomials $A_n(t)$, $A'_n(t)$ defined by these expansions
agree with those given in Examples~\ref{ex:gamman} and~\ref{ex:gammavn}.

Concerning $A_n(t)$, let $\tau=(r^2-rt)^{\frac 12}$; then
\[
\frac{1-r}{1-(2+t)r+2r^2}=\frac{1-r}{(1-r-i\tau)(1-r+i\tau)}
=\frac 12\left(\frac1{1-r-i\tau}+\frac1{1-r+i\tau}\right)\quad.
\]
The terms in the power series expansion of this expression are combinations of
powers of $(r-i\tau)$ and $(r+i\tau)$, so they may be obtained as the coefficients
of $\frac{s^k}{k!}$ in
\[
\frac 12\left( e^{(r+i\tau)s}+e^{(r-i\tau)s}\right) = e^{rs}\cdot \frac{e^{i\tau s}+e^{-i\tau s}}2 
= e^{rs} \cos(\tau s)\quad.
\]
This recovers the description of $A_n(t)$ given in Proposition~\ref{prop:GCgamman}.

The argument for $A'_n(t)$ is of course analogous. Again setting $\tau=(r^2-rt)^{\frac 12}$,
we have
\begin{align*}
\frac{(1-r)^2}{1-(2+t)r+2r^2} &= \frac 12\left(\frac{1-r}{1-r-i\tau}+\frac{1-r}{1-r+i\tau}\right) \\
&=\frac 12\left(\frac 1{1-i\frac{\tau}{1-r}}+\frac 1{1+i\frac{\tau}{1-r}}\right)
\end{align*}
and the terms in the power series expansion of this expression are the coefficients
of~$\frac{s^k}{k!}$ in
\[
\frac 12\left(e^{i\frac{\tau}{1-r}} + e^{-i\frac{\tau}{1-r}}\right)
=\cos\left( \frac{\tau}{1-r}\right)\quad,
\]
recovering the description of $A'_n(t)$ in Proposition~\ref{prop:GCgammanv}.
\end{proof}

\subsection{Graphs $\Gamma^v_n$ with arbitrary valence}
We will discuss some families of valence-$4$ vacuum graphs in~\S\ref{Mav}. 
The non-vacuum graphs to which the first formula applies have a natural generalization
for arbitrary valence: we can let $\Gamma^v_n$ be the graphs with melonic construction
\[
\big(((1,v-1,1),0,1),((1,v-1,1),1,2),((1,v-1,1),2,2),\dots, ((1,v-1,1),n-1,2)\big)
\]
for $v\ge 3$. For example, the graphs $\Gamma^3_n$ have the form
\begin{center}
\includegraphics[scale=.4]{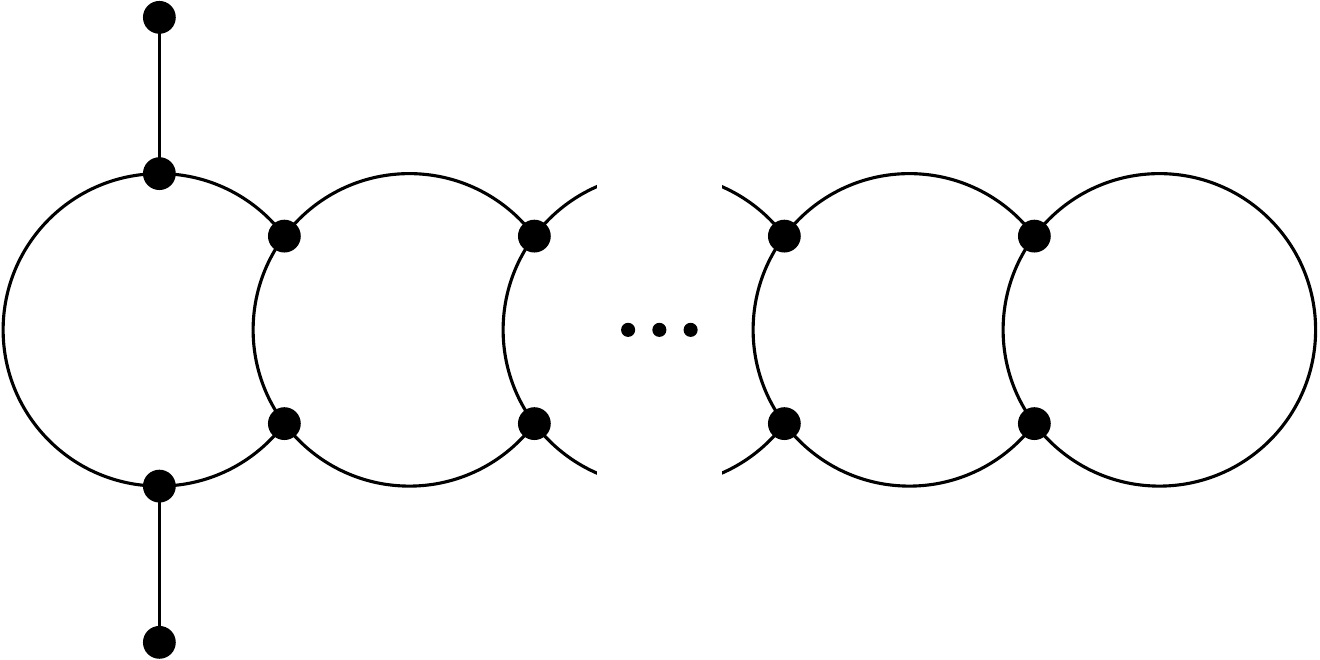}
\end{center}
while the graphs $\Gamma^5_n$ look like
\begin{center}
\includegraphics[scale=.4]{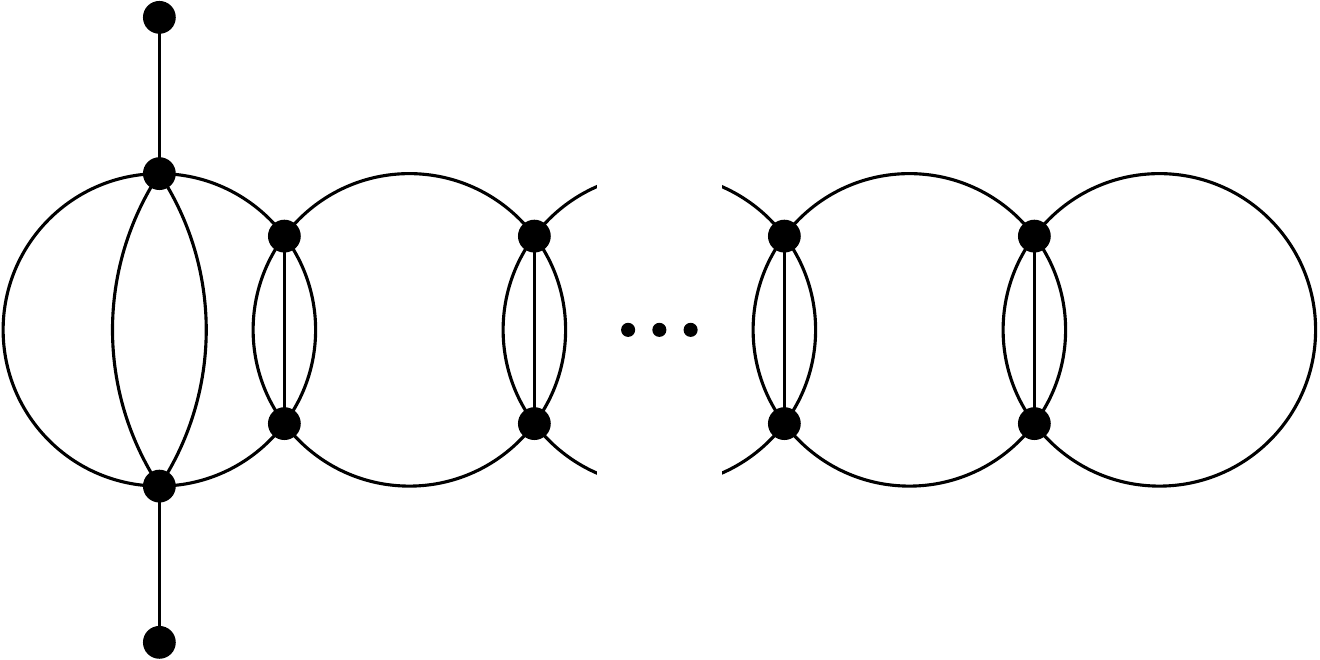}
\end{center}
The first several classes $\Ubb(\Gamma^3_n)$ are
\begin{align*}
n=1&:\qquad (\Tbb+1)^3\cdot \Tbb \\
n=2&:\qquad (\Tbb+1)^5\cdot (\Tbb^2+\Tbb) \\
n=3&:\qquad (\Tbb+1)^7\cdot (\Tbb^3+2\Tbb^2) \\
n=4&:\qquad (\Tbb+1)^9\cdot (\Tbb^4+3\Tbb^3+\Tbb^2) \\
n=5&:\qquad (\Tbb+1)^{11}\cdot (\Tbb^5+4\Tbb^4+3\Tbb^3) \\
n=6&:\qquad (\Tbb+1)^{13}\cdot (\Tbb^6+5\Tbb^5+6\Tbb^4+\Tbb^3) \\
n=7&:\qquad (\Tbb+1)^{15}\cdot (\Tbb^7+6\Tbb^6+10\Tbb^5+4\Tbb^4)
\end{align*}
It is natural to guess that for $n\ge 1$
\[
\Ubb(\Gamma^3_n)=(\Tbb+1)^{2n+1}\cdot C_n(\Tbb)
\]
with
\begin{equation}\label{eq:val3}
C_n(\Tbb)=\sum_{i=0}^n \binom i{n-i} \Tbb^i\quad.
\end{equation}
This may be proven by induction on the number of circles: the $m=2$ case of
formula~\eqref{eq:muledges} yields the recursion
\[
C_{n+1}=\Tbb\cdot (C_n+C_{n-1})\quad,
\]
which determines all $C_n$ from $C_1=\Tbb$, $C_2=\Tbb(\Tbb+1)$, 
confirming~\eqref{eq:val3}.
One can package this result as a generating function and draw the following
conclusion:

\begin{prop}
For $n\ge 1$,
\[
\Ubb(\Gamma^3_n)=(\Tbb+1)^{2n+1}\cdot\text{coefficient of $r^n$ in the
expansion of } \frac{1}{1-\Tbb r-\Tbb r^2}\quad.
\]
\end{prop}

A similar, but understandably more complex expression holds for arbitrary valence~$v$. 

\begin{prop}\label{prop:allv}
Let $v\ge 4$.
\begin{itemize}
\item
The class $\Ubb(\Gamma^v_n)$ is a multiple of $\Tbb^n(\Tbb+1)^{2n+1}$:
\[
\Ubb(\Gamma^v_n)=\Tbb^n(\Tbb+1)^{2n+1}\cdot A^v_n(\Tbb)
\]
for a polynomial $A^v_n(t)$ of degree $(v-3)n$.
\item
The polynomial $A^v_n(t)$ is the coefficient of $r^n$ in the series expansion of the
rational function $\alpha_n(r,t)=N(r,t)/D(r,t)$, where
\[
N(r,t)=\frac{1+t+((-1)^{v-3}-t^{v-3})\,r}{1+t}=1-\left(\sum_{i=0}^{v-4} (-1)^{v-i} t^i\right)r\qquad\text{and}
\]
\begin{multline*}
D(r,t)=1+\left(-vt^{v-3}-\sum_{i=0}^{v-3} (-1)^{v-i} (i+2)t^i\right)r \\
+\left((-1)^v T^{v-4}+\sum_{i=0}^{v-4} (-1)^{v-i}(v-3-i) t^{v-4+i}\right)r^2\quad.
\end{multline*}
\end{itemize}
\end{prop}

A formal proof of Proposition~\ref{prop:allv} may be constructed along the lines we
will provide explicitly for the case $v=4$ in \S\ref{Mp}. 

\begin{example}
Consider the case $v=10$; the rational function $\alpha_{10}(r,t)$ is
{\small\[
\frac{1-(1-t+t^2-t^3+t^4-t^5+t^6)r}
{1-(2-3t+4t^2-5t^3+6t^4-7t^5+8t^6+t^7)r
+(8t^6-6t^7+5t^8-4t^9+3t^{10}-2t^{11}+t^{12})r^2}
\]}
and the coefficient of $r^{13}$ in the series expansion of this rational function is
a polynomial of degree~$91$:
\[
t^{91}+103t^{90}+4794 t^{89}+\cdots-2455891878317453988 t^{45}+\cdots
+866304t^2-81920 t+4096\quad.
\]
According to Proposition~\ref{prop:allv}, the Grothendieck class for the melonic
graph constructed by
\[
(((1,9,1),0,1),((1,9,1),1,2),((1,9,1),2,2),\cdots, ((1,9,1),12,2))
\]
equals
\[
\Tbb^{13}(\Tbb+1)^{27}\cdot\big(\Tbb^{91}+103\Tbb^{90}+\cdots
-2455891878317453988 \Tbb^{45}+\cdots -81920 \Tbb+4096\big)\quad.
\]
This may be verified by applying the explicit recursion obtained in~\S\ref{MGc}.
\qede\end{example}

\section{Vacua}\label{Mav}

In this section we focus on melonic vacuum graphs. We first observe that there is a
close relation between Grothendieck classes of vacuum graphs and of related 
{\em non-\/}vacuum graphs. For this discussion, the graphs are not necessarily assumed
to be melonic; however, the result will explain melonic relations such as the one observed 
in~\eqref{eq:relvnv}.

\smallskip
\subsection{Vacuum and non-vacuum graphs relations}
Assume $\Gamma^v$ is a graph with a distinguished edge:
\begin{center}
\includegraphics[scale=.5]{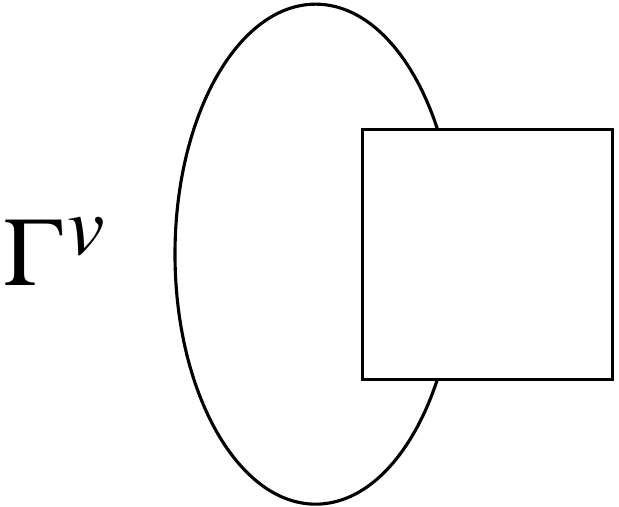}
\end{center}
and assume this edge is not a bridge in $\Gamma^v$. Consider two associated graphs:
the graph $\Gamma$ obtained by cutting the edge, and the graph $\overline \Gamma$
obtained by inserting a new edge crossing the given edge, with vertices as indicated:
\begin{center}
\includegraphics[scale=.5]{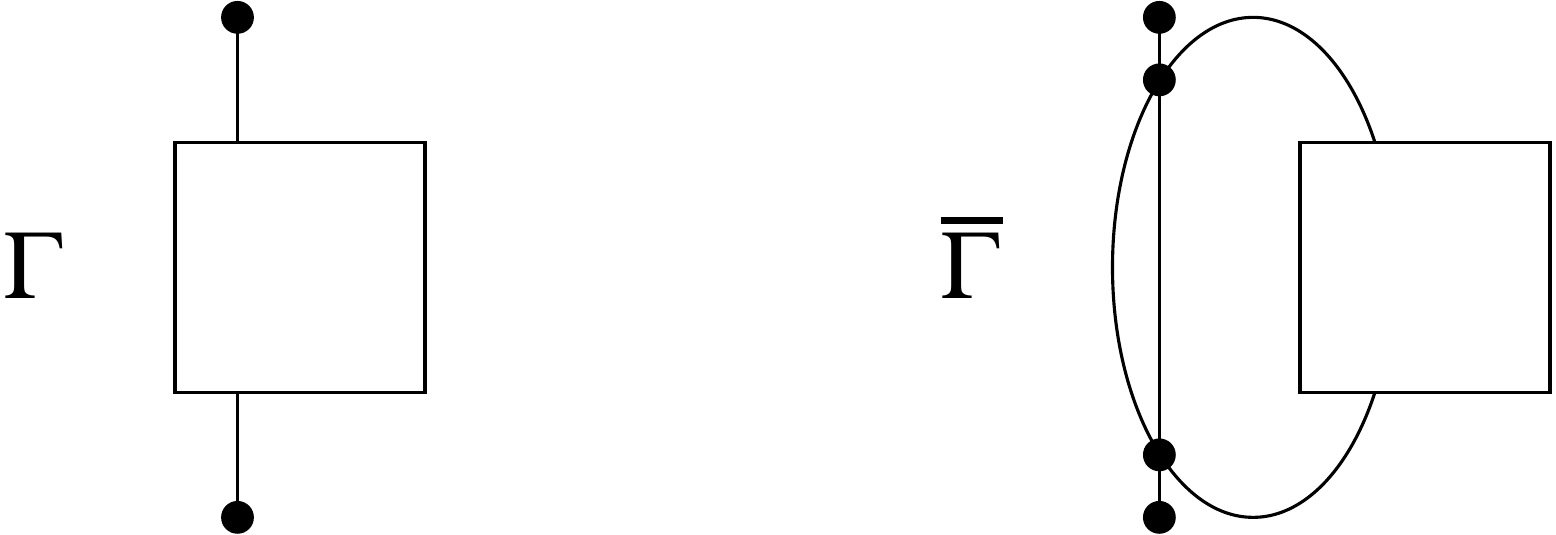}
\end{center}

\begin{lemma}\label{lemma:vnv}
\[
\Ubb(\Gamma^v)=\frac{\Ubb(\overline \Gamma)-\Tbb(\Tbb+1)^2 \Ubb(\Gamma)}{\Tbb(\Tbb+1)^4}
\]
\end{lemma}

\begin{proof}
This is an application of the formula for the effect on Grothendieck classes of adding one 
parallel edge to a given (non-bridge, non-looping) edge in a graph, i.e., the case $m=2$
of~\eqref{eq:muledges}. Place two valence-$2$ vertices on the 
joined edge in $\Gamma^v$, creating an edge $e$ in a graph $\Gamma'$; by construction,
$e$ is neither a bridge nor a looping edge. Then
\[
\Ubb(\Gamma')=(\Tbb+1)^2 \Ubb(\Gamma^v)\quad,\quad
\Ubb(\Gamma'/e)=(\Tbb+1) \Ubb(\Gamma^v)\quad,
\]
while $\Gamma'\smallsetminus e=\Gamma$. Replacing $e$ by two parallel edges
produces $\overline\Gamma$ without the two external edges. 
Applying~\eqref{eq:muledges} then gives
\[
\frac{\Ubb(\overline \Gamma)}{(\Tbb+1)^2} = f_2 (\Tbb+1)^2 \Ubb(\Gamma^v)
+g_2 (\Tbb+1) \Ubb(\Gamma^v) + h_2 \Ubb(\Gamma)\quad,
\]
that is (cf.~\eqref{eq:coefs})
\[
\frac{\Ubb(\overline \Gamma)}{(\Tbb+1)^2} = \Tbb (\Tbb+1)^2 \Ubb(\Gamma^v)
+ \Tbb \Ubb(\Gamma)\quad,
\]
with the stated result.
\end{proof}

In the applications we have in mind, $\Gamma$ may be a melonic non-vacuum graphs
constructed by
\[
((1,a_2,\dots, a_{r-1},1),0,1), t_2,\dots, t_n\quad;
\]
the graph $\Gamma^v$ will then be the (melonic) vacuum graph obtained by joining the 
two valence-$1$ vertices of $\Gamma$, and $\overline \Gamma$ is the non-vacuum graph 
constructed by
\[
((1,3,1),0,1), ((1,a_2,\dots, a_{r-1},1),1,2), t'_2,\dots, t'_n
\]
where $t'_i=(b_i,p_i+1,k_i)$ if $t_i=(b_i,p_i,k_i)$, $i=2,\dots,n$.
Lemma~\ref{lemma:vnv} shows that the class $\Ubb(\Gamma^v)$ of the vacuum graph 
is determined by the classes $\Ubb(\Gamma)$, $\Ubb(\overline\Gamma)$ of the 
associated non-vacuum graphs.

For example, with notation as in~\S\ref{MEcacI}, Lemma~\ref{lemma:vnv} implies that
\[
\Ubb(\Gamma'_{n+1})=\frac{\Ubb(\Gamma_{n+1})
-\Tbb(\Tbb+1)^2 \Ubb(\Gamma_n)}{\Tbb(\Tbb+1)^4}\quad;
\]
with $\Ubb(\Gamma_n)=\Tbb^n(\Tbb+1)^{2n+1} A_n(\Tbb)$ and $\Ubb(\Gamma'_n)
=\Tbb^{n-1} (\Tbb+1)^{2n-3}A'_n(\Tbb)$ as in~\S\ref{MEcacI}, this relation gives
\[
\Tbb^n (\Tbb+1)^{2n-1}A'_{n+1}(\Tbb)
=\frac{\Tbb^{n+1}(\Tbb+1)^{2n+3} A_{n+1}(\Tbb)
-\Tbb(\Tbb+1)^2 \Tbb^n(\Tbb+1)^{2n+1} A_n(\Tbb)}
{\Tbb(\Tbb+1)^4},
\]
that is,
\[
A'_{n+1}(\Tbb)=A_{n+1}(\Tbb)-A_n(\Tbb)\quad,
\]
and this proves~\eqref{eq:relvnv}. 

\smallskip
\subsection{Tree structure for valence-four vacua}
Next, we consider specifically melonic vacuum graphs in which every vertex has
valence~$4$. These graphs may be drawn as loopless unions of ovals:
\begin{center}
\includegraphics[scale=.5]{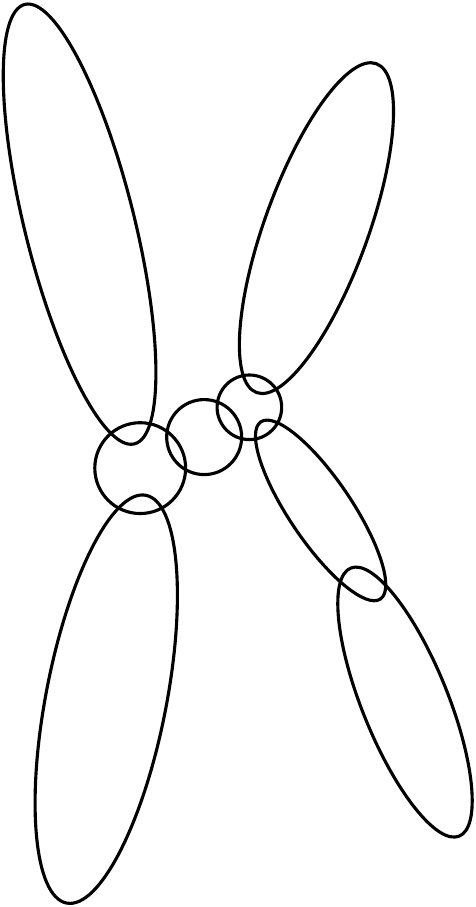}
\end{center}
The information carried by a vacuum melonic graph in valence~$4$ is equivalent to 
the information of a tree, for example the tree
\begin{center}
\includegraphics[scale=.5]{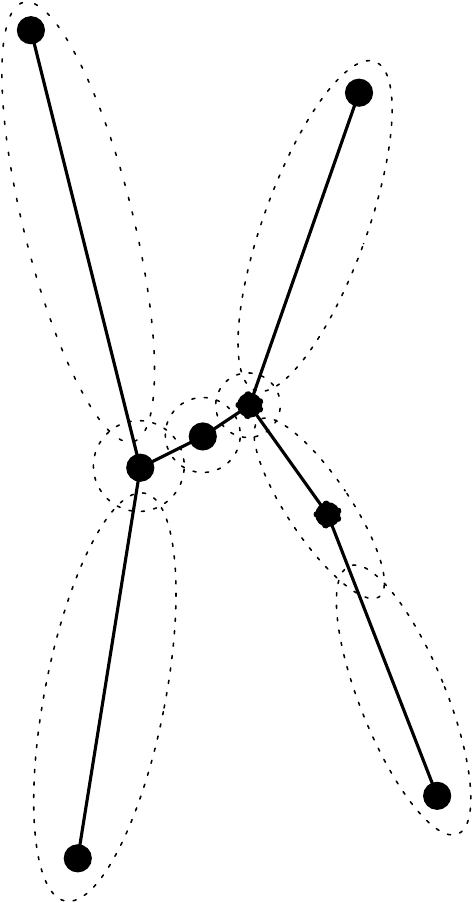}
\end{center}
for the graph shown above. 
Every node of this tree corresponds to one of the ovals,
and two nodes are connected by an edge if and only if the corresponding ovals meet.
Given a tree, a corresponding melonic construction is obtained in the evident way
by associating one arbitrary edge of the tree with a $4$-banana and labeling the
other edges with appropriate $((1,3,1),*,*)$ tuples as prescribed by adjacencies in 
the tree. For example, the edges of the above tree could be marked as follows
(where we also numbered the edge of the tree to reflect the stage of the
corresponding tuple in the melonic construction; many other choices are possible):
\begin{center}
\includegraphics[scale=.5]{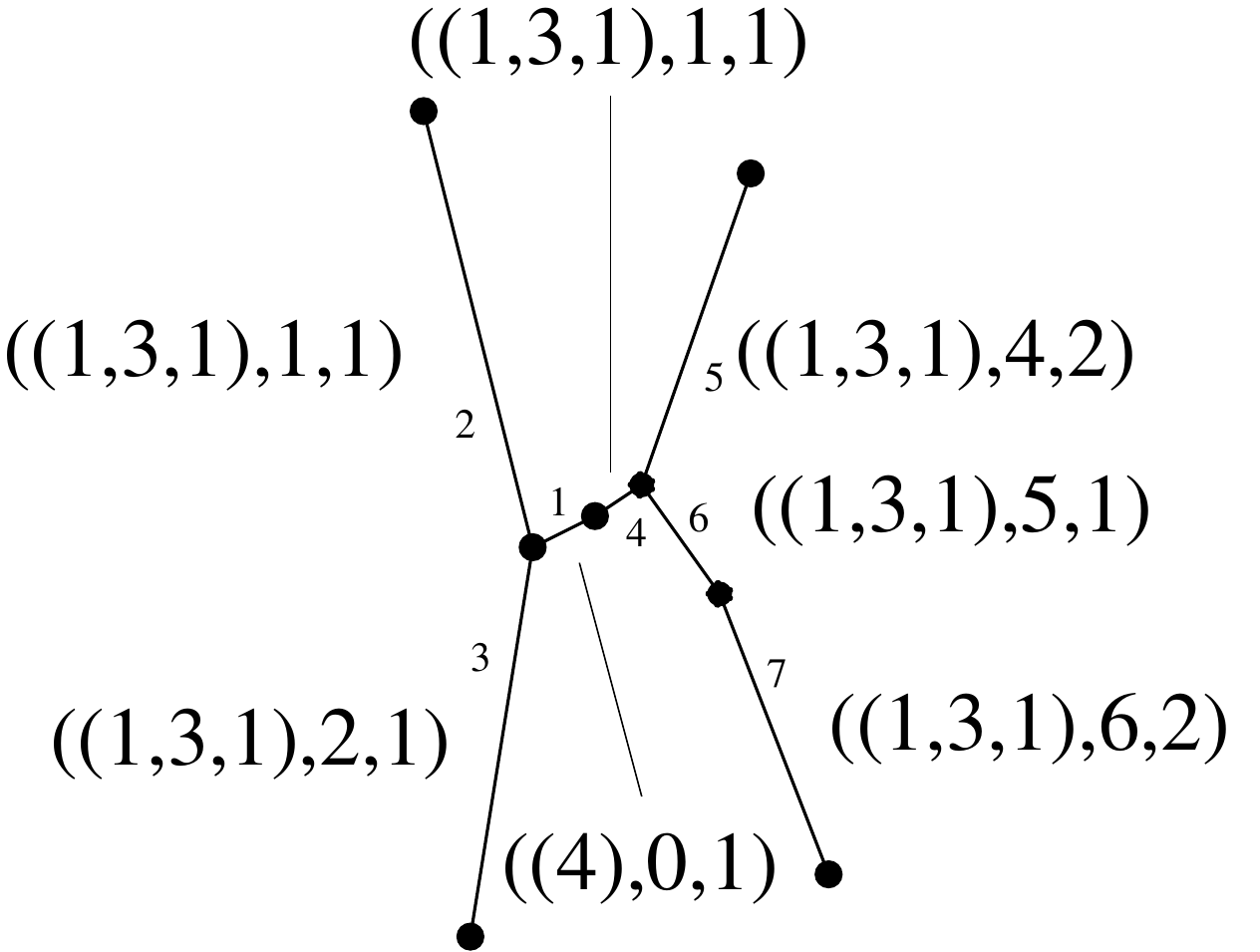}
\end{center}
leading to the melonic construction
\begin{multline}\label{eq:melcoor}
((4),0,1), ((1,3,1),1,1), ((1,3,1),2,1), ((1,3,1),1,1), \\
((1,3,1),4,2), ((1,3,1),5,1), ((1,3,1),6,2)\quad.
\end{multline}

Alternatively, one could label one node of the tree by a $2$-banana and the 
remaining nodes by $((1,3,1),*,*)$ tuples; the corresponding construction will
produce a vacuum melonic graph with two extra valence-$2$ vertices.
This strategy is used below in Example~\ref{ex:vstars}.

\smallskip
\subsection{Recursion relations for vacuum bubbles}
It is natural to ask whether a simple recursion may exist between the Grothendieck
classes of vacuum melonic graphs, reflecting the tree-like structure underlying them. 
The only instance known to us of such a recursion goes as follows. Assume a branch
of the tree projects out of the main body; let $U_n$ denote the Grothendieck class
of the vacuum melonic graph obtained by adding $n$ edges to such a branch.
\begin{center}
\includegraphics[scale=.5]{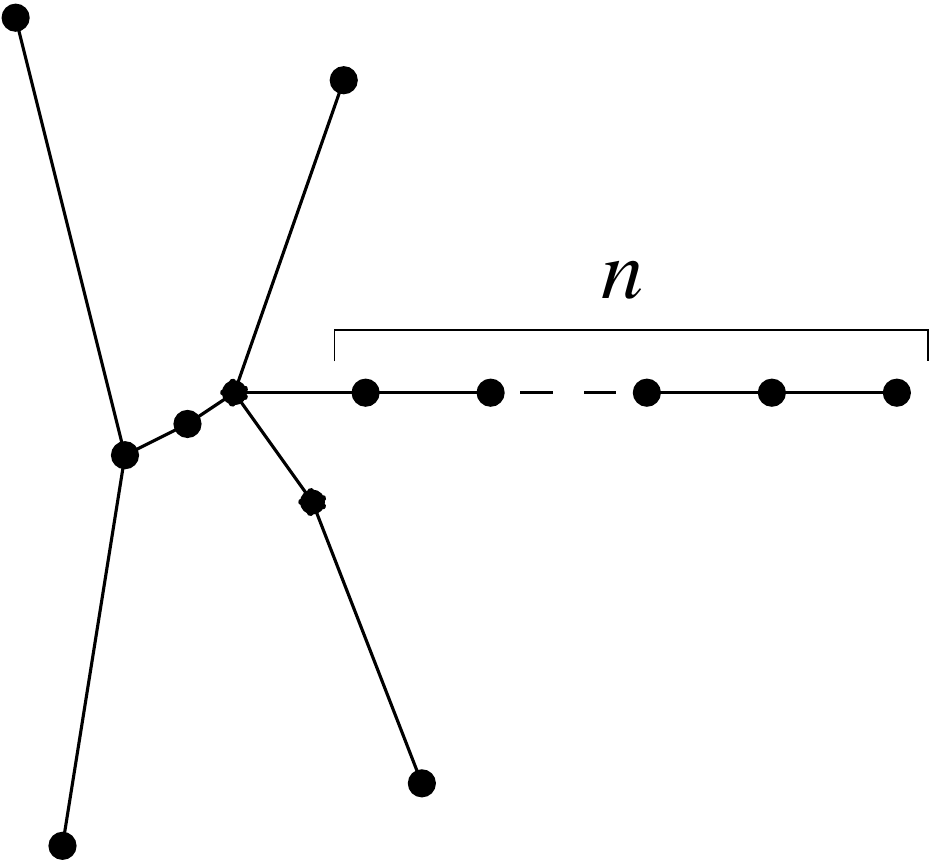}
\end{center}

\begin{claim}\label{claim:recurbra}
For $n\ge 2$,
\begin{equation}\label{eq:recurbra}
U_{n+1} = \Tbb(\Tbb+1)^2(\Tbb+2)\, U_n -2\Tbb^2(\Tbb+1)^4\, U_{n-1}\quad.
\end{equation}
\end{claim}

We will prove this formula in~\S\ref{Mp}; in fact we will prove that this formula
holds even if the starting graph is not melonic. This will be our main tool in the proofs
of the propositions stated thus far, as well as Proposition~\ref{prop:stars}, stated
below.

\begin{example}\label{ex:vstars}
Let $\Sigma^s_n$ be the vacuum melonic graph corresponding to the star-shaped
tree
\begin{center}
\includegraphics[scale=.4]{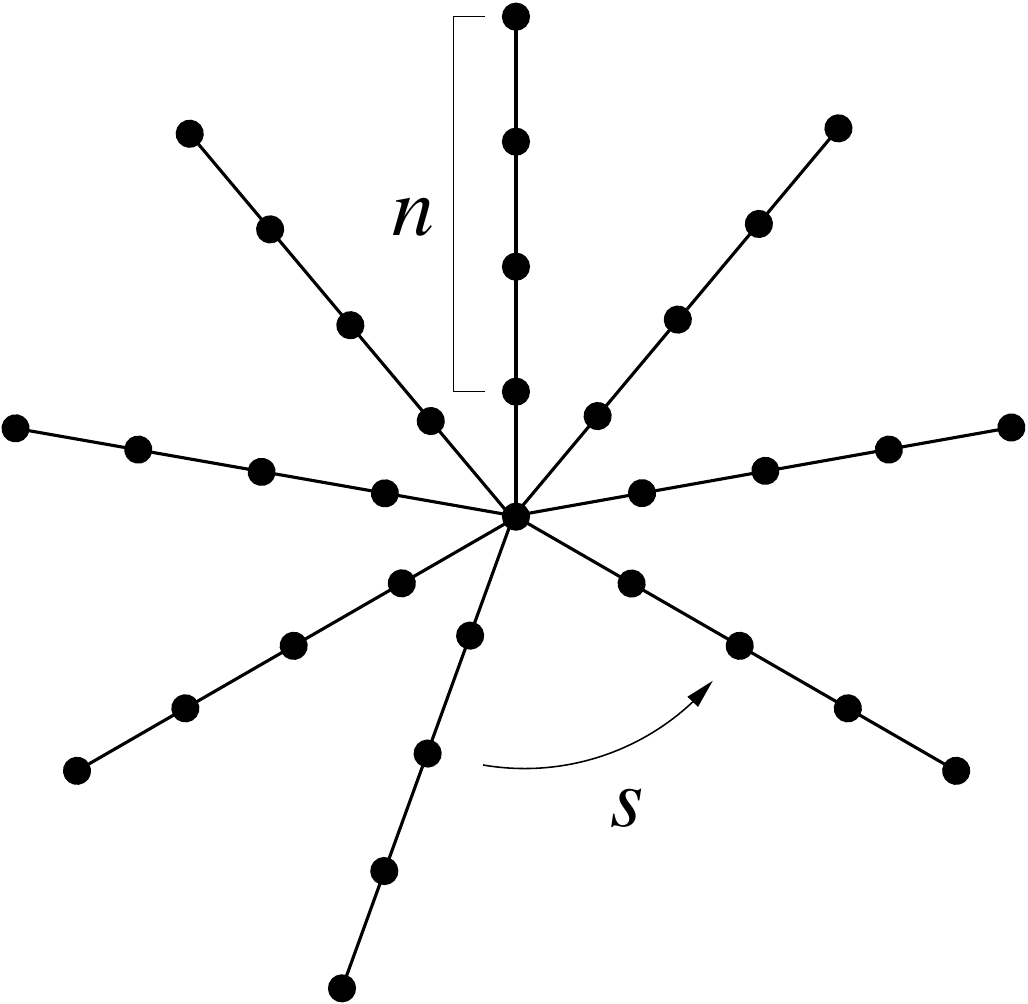}
\end{center}
with $s$ rays and $n$ nodes along each ray. For example, $\Sigma^3_4$ is the
following melonic vacuum star: 
\begin{center}
\includegraphics[scale=.3]{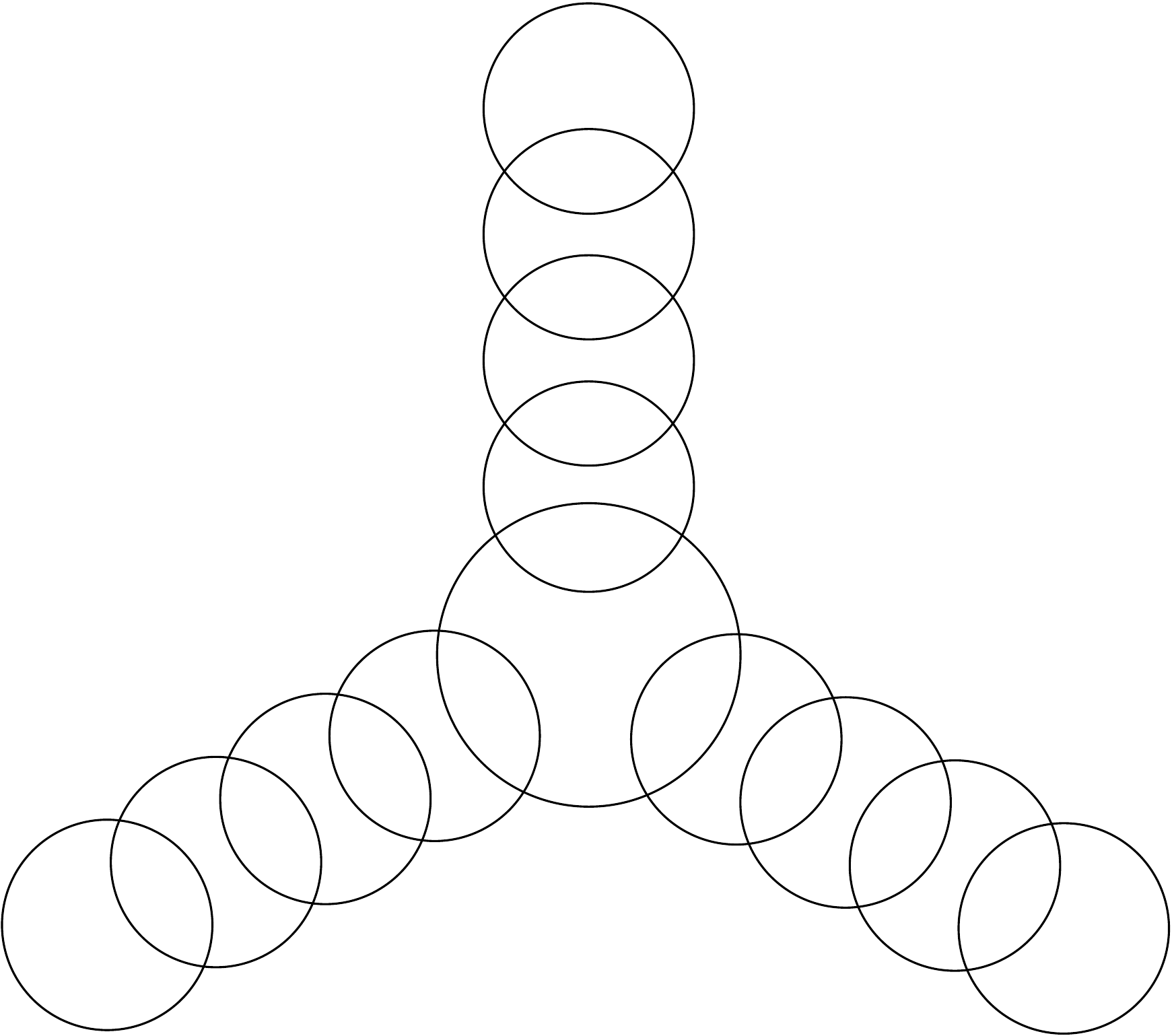}
\end{center}
Interpreting the central node as a $2$-banana (thus adding two valence-$2$ vertices
to the corresponding circle) leads to the following melonic construction for
$\Sigma^s_n$:
\begin{align*}
((2),0,1), &((\underbrace{1,3,1,\dots, 3,1}_{\text{$s$ `$3$'}}),1,1),\\
&((1,3,1),2,2),\dots, ((1,3,1),2,2s), \\
&((1,3,1),3,2),\dots, ((1,3,1),2+s,2), \\
&((1,3,1),3+s,2),\dots, ((1,3,1),2+2s,2), \\
& \qquad\dots, \\
&((1,3,1),3+(n-3)s,2),\dots, ((1,3,1),2+(n-2)s,2)
\end{align*}
(Of course many alternatives are possible.)
This construction may be used to compute Grothendieck classes in specific 
examples, by using the recursion obtained in~\S\ref{MGc}
(and dividing by $(\Tbb+1)^2$ to account for the two additional valence-$2$ 
vertices arising in the construction). On the basis of 
extensive data, one can formulate the following statement.

\begin{prop}\label{prop:stars}
Let $\sigma^s_n(t)$ be the polynomials defined by the expansion
\[
\frac{1-2r+((s-1)t-(s-2))r^2}{1-(2+t)r+2r^2}=
1+\sum_{n\ge 0} \sigma^s_n(t)\, r^{n+1}\quad.
\]
Then for $s,n\ge 1$
\[
\Ubb(\Sigma^s_n) = \Tbb^{sn} (\Tbb+1)^{2sn-1} A_n(\Tbb)^{s-1} \sigma^s_n(\Tbb)\quad,
\]
where $A_n(t)$ is the polynomial appearing in Proposition~\ref{prop:GCgamman}.
\end{prop}

For example, according to the above definition,
\[
\sigma^{11}_6(t)=t^7+22t^6+139t^5+290t^4-8t^3-424t^2-44t+88\quad,
\]
and one finds
\begin{multline*}
\Tbb^{66} (\Tbb+1)^{131} A_6(\Tbb)^{10} \sigma^{11}_6(\Tbb) \\
=\Tbb^{264}+263\, \Tbb^{263}+34211\, \Tbb^{262}+2935019\, \Tbb^{261}
+\cdots+26065315469197312\,\Tbb^{76}\quad,
\end{multline*}
matching the result of the computation of the Grothendieck class $\Ubb(\Sigma^{11}_6)$
\begin{center}
\includegraphics[scale=.5]{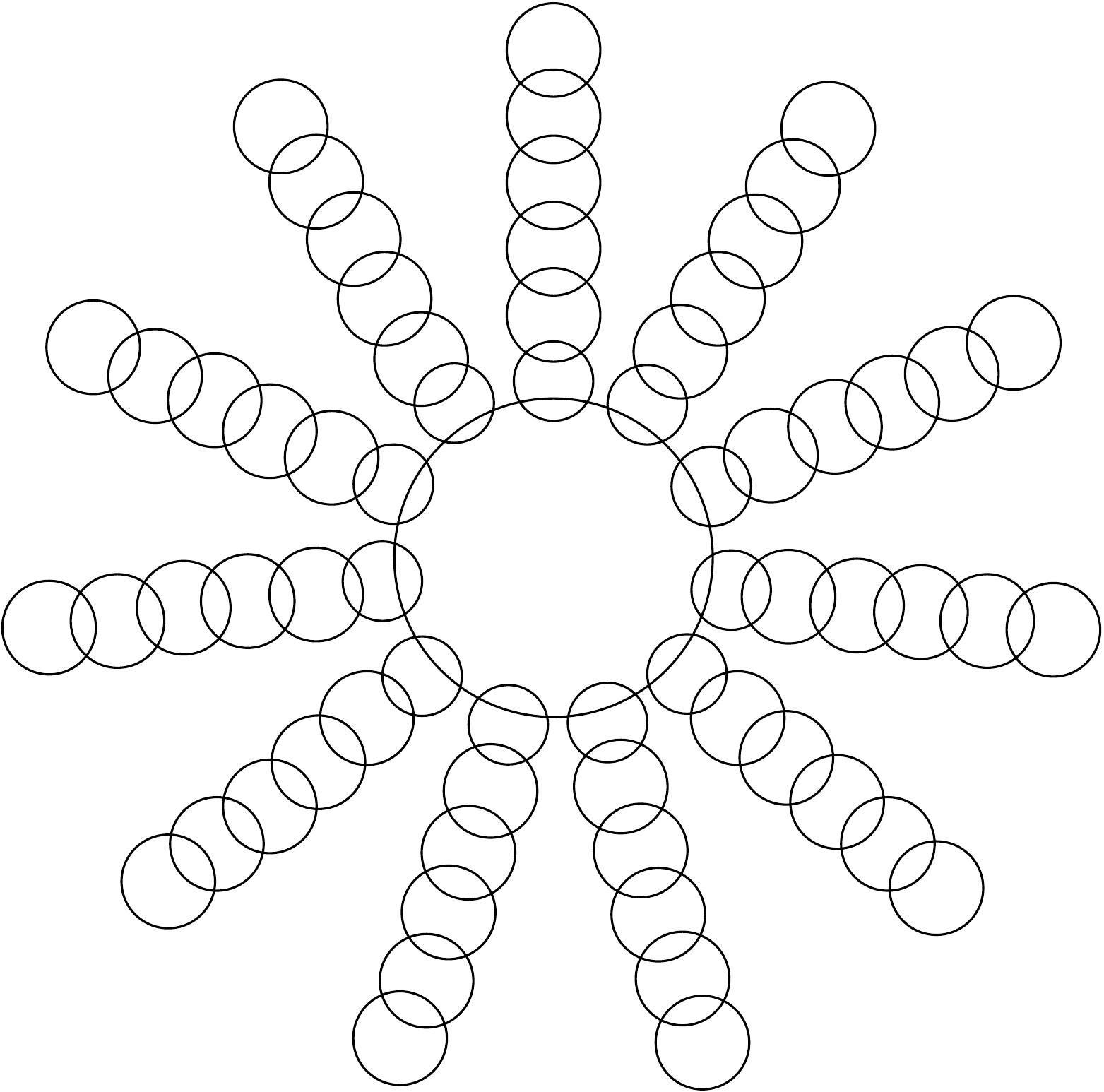}
\end{center}
by means of the basic recursion obtained in~\S\ref{MGc}.
\qede\end{example}

The proof of Proposition~\ref{prop:stars} is given in~\S\ref{Mav}.
We record the following consequence, which calls for a more geometric explanation.
The relation \eqref{Sigmasigma} below suggests that 
the complement of the hypersurface $\hat X_{\Sigma^s_n}$
may be realized as a fibration over products of complements of $\hat X_{\Gamma_n}$.
This suggests the possible presence of interesting geometric relations between these 
families of graph hypersurfaces. 

\begin{corol}\label{cor:diviSi}
If $s\le 2n$, then $\Ubb(\Gamma_n)^{s-1}$ divides $\Ubb(\Sigma^s_n)$.
More precisely,
\begin{equation}\label{Sigmasigma}
\Ubb(\Sigma^s_n) = \Tbb^n (\Tbb+1)^{2n-s} \Ubb(\Gamma_n)^{s-1} \sigma^s_n(\Tbb)\quad.
\end{equation}
\end{corol}

\begin{proof}
The given equality follows from the formula given in Proposition~\ref{prop:stars} and
the expression for $\Ubb(\Gamma_n)$ obtained in Proposition~\ref{prop:GCgamman}.
\end{proof}

\begin{remark}\label{rem:diviex}
Corollary~\ref{cor:diviSi} implies the divisibility relation~\eqref{eq:divi} observed 
in~\S\ref{MEcacI}. Indeed, the graph $\Sigma^2_n$ consists of a string of $2n+1$
circles:
\begin{center}
\includegraphics[scale=.35]{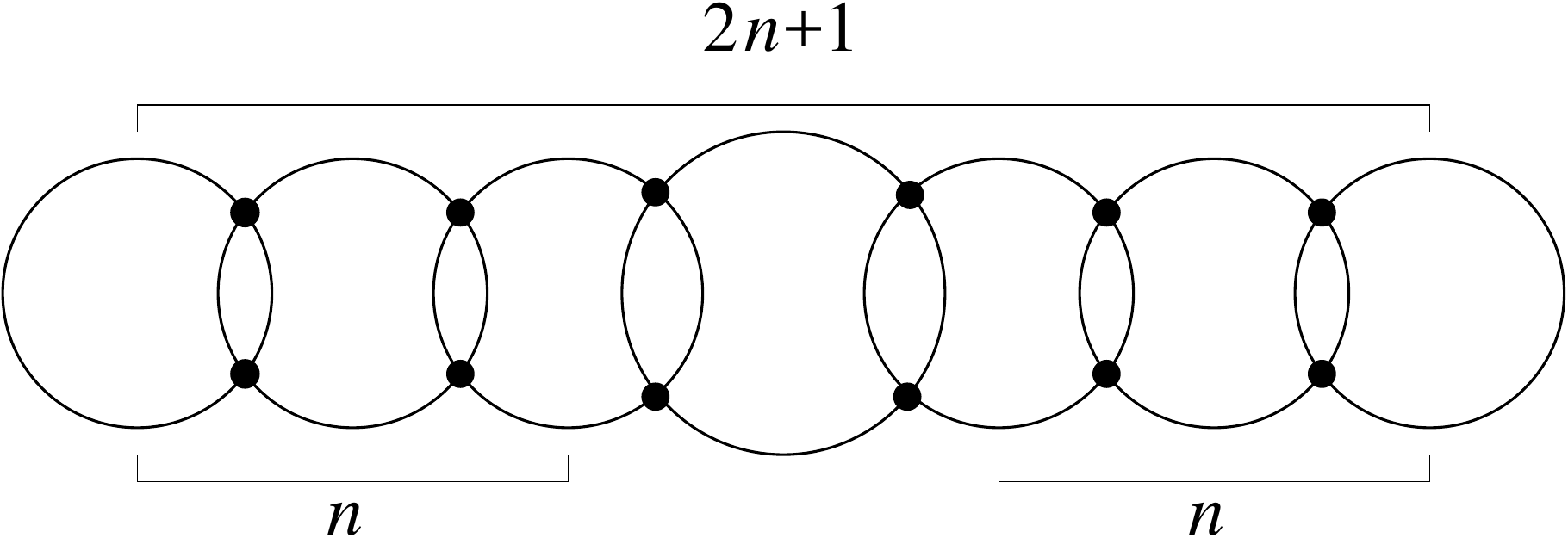}
\end{center}
That is, $\Sigma^2_n=\Gamma'_{2n+1}$, with notation as in~Example~\ref{ex:gammavn}.
For this graph, Corollary~\ref{cor:diviSi} states that $\Ubb(\Gamma_n)$ divides 
$\Ubb(\Gamma'_{2n+1})$, and this is precisely the assertion in~\eqref{eq:divi}.
\qede\end{remark}

\section{Proofs}\label{Mp}

Proposition~\ref{prop:GCgamman} and Proposition~\ref{prop:GCgammanv} will be
proved in the equivalent form presented in~Proposition~\ref{prop:both}. 
For clarity we will focus on the case of valence~$4$ given in these propositions; 
the same method could be used to prove Proposition~\ref{prop:allv}. 

The statement we will prove will actually be substantially more general than 
Propositions~\ref{prop:GCgamman} and~\ref{prop:GCgammanv}: it consists
of a recursion ruling the Grothendieck classes of graphs obtained by extending
any given graph by a tower of $3$-bananas.

Let $G$ be a (not necessarily melonic) graph, and let $e$ be an edge of $G$. 
Let $G_n$ be the graph obtained by applying a chain of $(1,3,1)$-bananifications
starting from~$e$: 
\begin{center}
\includegraphics[scale=.4]{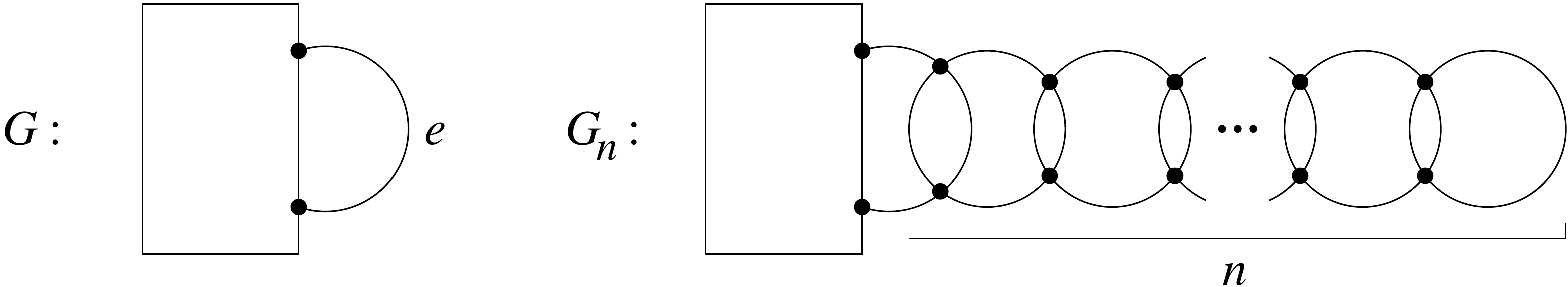}
\end{center}

\begin{thm}\label{thm:rationG}
The generating function for the Grothendieck classes $\Ubb(G_n)$ is rational, 
with denominator independent of $G$. More precisely, there exists a polynomial 
$P(\Tbb,\rho)$ with integer coefficients such that
\[
\sum_{n\ge 0} \Ubb(G_n) \rho^n = \frac{P(\Tbb,\rho)}
{1-\Tbb (\Tbb+1)^2 (\Tbb+2)\,\rho+2 \Tbb^2 (\Tbb+1)^4\,\rho^2}
\]
\end{thm}

This statement focuses on the fact that the generating function is rational, and gives
an explicit form for its denominator, which depends on the bananification process itself
rather than on the graph $G$. The graph $G$ determines the numerator $P(\Tbb,\rho)$;
precise formulas will be given below in Theorem~\ref{thm:rationG2}. 
In practice, Theorem~\ref{thm:rationG} and a few case-by-case explicit computations 
of $\Ubb(G_n)$ for low values of $n$ determine $P(\Tbb,\rho)$.

\begin{proof}
Denote by $H_n$ the graph obtained from $G_n$ by replacing
the last $3$-bananification with a $2$-bananification:
\begin{center}
\includegraphics[scale=.4]{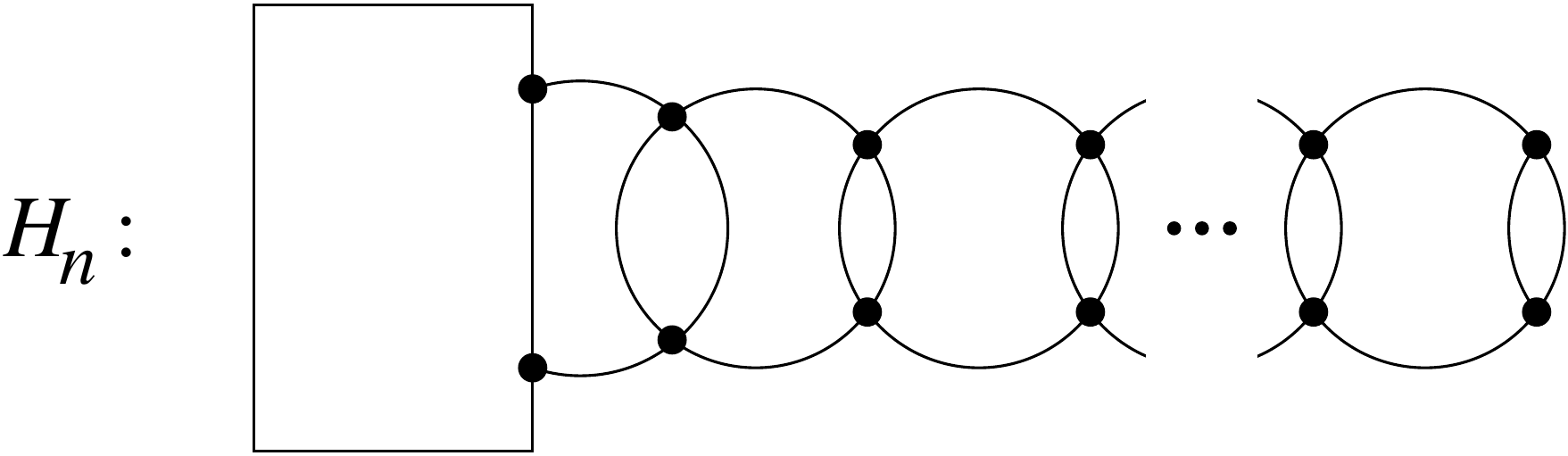}
\end{center}
Denote $\Ubb(G_n)$ by $U_n$, $\Ubb(H_n)$ by $V_n$.
Assume $n\ge 2$. Consider the graph $G''$ obtained by splitting one of the parallel
edges of the top banana in $G_{n-1}$ into three edges; let $e'$ be the central edge
so produced, and note that $e'$ is not a bridge or a looping edge of $G''$.
\begin{center}
\includegraphics[scale=.4]{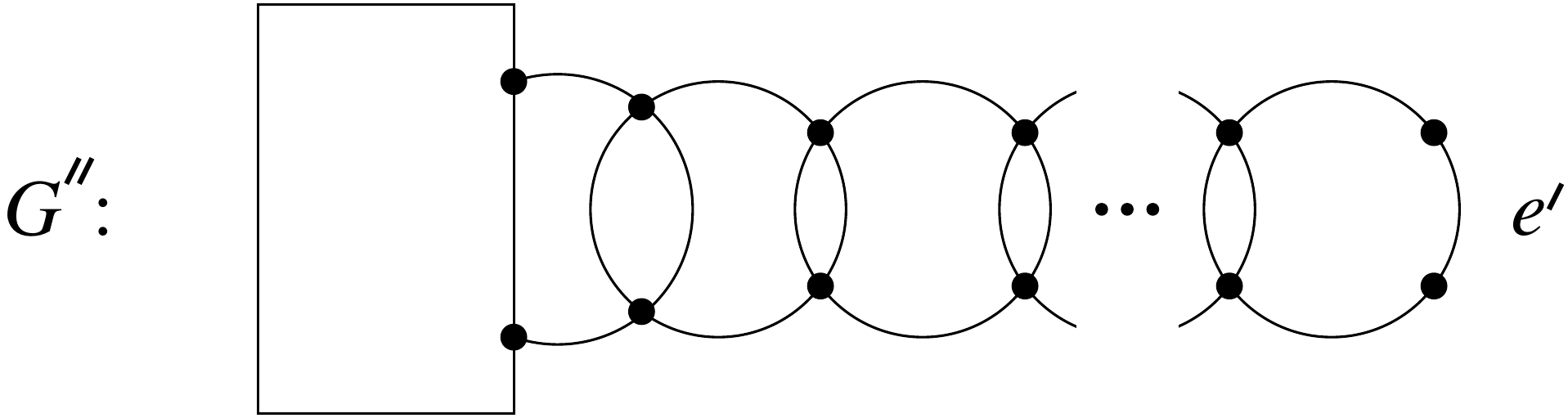}
\end{center}
The contraction $G':=G''/ e'$ may be obtained from $G_{n-1}$
by splitting the same edge of the top banana into {\em two\/} edges, and the 
deletion $H'=G''\smallsetminus e'$ may be obtained from $H_{n-1}$ by attaching 
two external edges to the vertices of the top ($2$-)banana:
\begin{center}
\includegraphics[scale=.4]{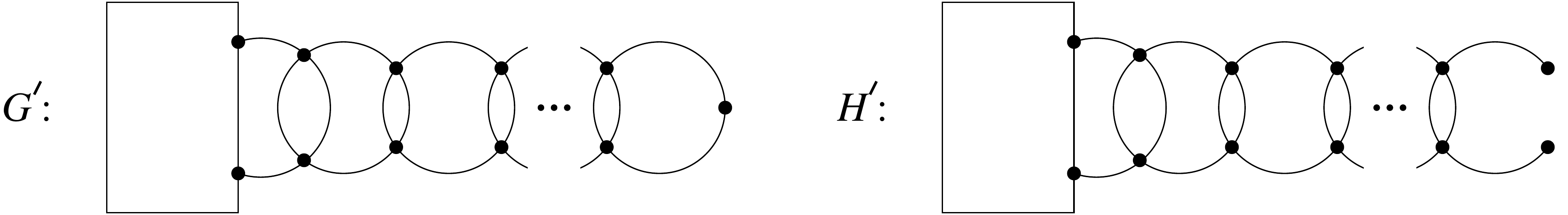}
\end{center}
By~\eqref{eq:muledges}, we have
\begin{align*}
V_n&=f_2\,\Ubb(G'')+g_2\,\Ubb(G')+h_2\,\Ubb(H') \\
U_n&=f_3\,\Ubb(G'')+g_3\,\Ubb(G')+h_3\,\Ubb(H') \quad,
\end{align*}
where $f_2,f_3$, etc., are as in~\eqref{eq:coefs}. We now note that
\[
\left\{
\aligned
\Ubb(G'')&=(\Tbb+1)^2 \Ubb(G_{n-1})= (\Tbb+1)^2 U_{n-1}\\
\Ubb(G')&=(\Tbb+1)\,\, \Ubb(G_{n-1})= (\Tbb+1)\,\, U_{n-1}\\
\Ubb(H')&=(\Tbb+1)^2 \Ubb(H_{n-1})= (\Tbb+1)^2 V_{n-1}
\endaligned
\right.\quad;
\]
further,
\[
f_2(\Tbb+1)^2 +g_2 (\Tbb+1)=\Tbb (\Tbb+1)^2\quad,\quad
f_3(\Tbb+1)^2 +g_3 (\Tbb+1)=\Tbb (\Tbb+1)^3
\]
while $h_2=\Tbb$, $h_3=(\Tbb-1)\Tbb$. The above formulas can then be rewritten
\begin{equation}\label{qe:inthep}
\begin{aligned}
V_n &= \Tbb(\Tbb+1)^2\, U_{n-1} + \Tbb (\Tbb+1)^2\, V_{n-1} \\
U_n &= \Tbb(\Tbb+1)^3\, U_{n-1} + (\Tbb-1)\Tbb (\Tbb+1)^2\, V_{n-1} \quad.
\end{aligned}
\end{equation}
These imply
\begin{align*}
(\Tbb-1) V_n &= (\Tbb-1)\Tbb(\Tbb+1)^2\, U_{n-1} + (\Tbb-1)\Tbb (\Tbb+1)^2\, V_{n-1} \\
&= (\Tbb-1)\Tbb(\Tbb+1)^2 \, U_{n-1} + \big(U_n-\Tbb(\Tbb+1)^3\, U_{n-1}\big) \\
&= U_n-2\Tbb(\Tbb+1)^2\,U_{n-1}
\end{align*}
and therefore
\begin{align*}
U_{n+1} &= \Tbb(\Tbb+1)^3\, U_n + (\Tbb-1)\Tbb (\Tbb+1)^2\, V_n \\
&= \Tbb(\Tbb+1)^3 \, U_n + \Tbb (\Tbb+1)^2\, 
\left(U_n-2\Tbb(\Tbb+1)^2\,U_{n-1}\right) \\
&= \Tbb(\Tbb+1)^2(\Tbb+2)\, U_n -2\Tbb^2(\Tbb+1)^4\, U_{n-1}\quad.
\end{align*}
(This proves Claim~\ref{claim:recurbra}.)
Now, for $n\ge 2$, the coefficient of $\rho^{n+1}$ in the product
\[
\big(1-\Tbb (\Tbb+1)^2 (\Tbb+2)\,\rho+2 \Tbb^2 (\Tbb+1)^4\,\rho^2\big)\cdot 
\sum_{n\ge 0} U_n\,\rho^n
\]
equals
\[
U_{n+1}-\Tbb(\Tbb+1)^2(\Tbb+2)\, U_n+2\Tbb^2(\Tbb+1)^4\, U_{n-1}=0\quad.
\]
This product is therefore a polynomial $P(\Tbb,\rho)$, and this proves the 
statement.
\end{proof}

The argument shows that
\begin{equation}\label{eq:pofT}
\begin{aligned}
P(\Tbb,\rho) &= \big(1-\Tbb (\Tbb+1)^2 (\Tbb+2)\,\rho+2 \Tbb^2 (\Tbb+1)^4\,\rho^2\big)\cdot 
\sum_{n\ge 0} U_n\,\rho^n \\
&=\Ubb(G)+(\Ubb(G_1)-\Tbb(\Tbb+1)^2 (\Tbb+2)\,\Ubb(G))\,\rho \\
&\qquad+\big(\Ubb(G_2)-\Tbb(\Tbb+1)^2(\Tbb+2)\, \Ubb(G_1)+2\Tbb^2(\Tbb+1)^4\, \Ubb(G)
\big)\,\rho^2\quad.
\end{aligned}
\end{equation}
If $e$ is not a bridge, then the argument proves the same recursion
for $n\ge 1$; it follows that the coefficient of $\rho^2$ in $P(\Tbb,\rho)$ is $0$ in this case.
Maybe a little surprisingly, the same conclusion holds if $e$ {\em is\/} a bridge
(as we will prove below);
thus, the polynomial $P(\Tbb,\rho)$ is of degree~$1$ in $\rho$.
This polynomial is determined by $\Ubb(G)$ and the deletion~$\Ubb(G\smallsetminus e)$,
as we will see below.

In fact, Theorem~\ref{thm:rationG} and the direct computation of a few values of $\Ubb(G_n)$
suffice to determine the numerator.

\begin{example}
The melonic valence-$4$ vacuum graphs corresponding to the trees
\begin{center}
\includegraphics[scale=.5]{branchoutor}
\end{center}
have melonic construction obtained by extending~\eqref{eq:melcoor}:
\begin{multline*}
((4),0,1), ((1,3,1),1,1), ((1,3,1),2,1), ((1,3,1),1,1), \\
((1,3,1),4,2), ((1,3,1),5,1), ((1,3,1),6,2), \\
((1,3,1),6,1),((1,3,1),8,2),\dots,((1,3,1),n+6,2)
\quad.
\end{multline*}
Using the recursion obtained in~\S\ref{MGc}, we can compute the following 
Grothendieck classes:
\begin{align*}
n=0&:\quad 
\Tbb^7(\Tbb+1)^{14}(\Tbb^7+13\Tbb^6+56\Tbb^5+80\Tbb^4-17\Tbb^3-77\Tbb^2+8) \\
n=1&:\quad
\Tbb^8(\Tbb+1)^{17}(\Tbb^7+14\Tbb^6+64\Tbb^5+94\Tbb^4-29\Tbb^3-100\Tbb^2+12\Tbb+8) \\
n=2&:\quad 
\Tbb^{10}(\Tbb+1)^{18}(\Tbb+3)(\Tbb^7+14\Tbb^6+64\Tbb^5+96\Tbb^4-19\Tbb^3-102\Tbb^2-6\Tbb+16)
\end{align*}
and this is (more than) enough information to determine $P(\Tbb,\rho)$: if 
$U_0,U_1,U_2$ are these three classes, the product 
\[
\big(1-\Tbb (\Tbb+1)^2 (\Tbb+2)\,\rho+2 \Tbb^2 (\Tbb+1)^4\,\rho^2\big)\cdot
(U_0+U_1\,\rho+U_2\,\rho^2) 
\]
equals
\begin{multline*}
\Tbb^7(\Tbb+1)^{14}(\Tbb^7+13\Tbb^6+56\Tbb^5+80\Tbb^4-17\Tbb^3-77\Tbb^2+8)\\
-2\,\Tbb^8(\Tbb+1)^{16}(2\Tbb^6+17\Tbb^5+39\Tbb^4+9\Tbb^3-33\Tbb^2-6\Tbb+4)\,\rho
\end{multline*}
modulo $\rho^3$. As expected, the coefficient of $\rho^2$ vanishes. The
polynomial $P(\Tbb,\rho)$ must equal this degree~$1$ polynomial in $\rho$.
\qede\end{example}

In general, $P(\Tbb,\rho)$ is determined by the Grothendieck classes of $G$ 
and (if $e$ is not a bridge) $G\smallsetminus e$, if the latter is known.

\begin{thm}\label{thm:rationG2}
With notation as above, let $r=\Tbb(\Tbb+1)^2\, \rho$. Then we have
\[
\sum_{n\ge 0} \Ubb(G_n) \rho^n = \frac{1-r}
{1- (\Tbb+2)\,r+2\, r^2}\cdot\Ubb(G)
\]
if $e$ is a bridge in $G$, and
\[
\sum_{n\ge 0} \Ubb(G_n)\, \rho^n = \frac{\Ubb(G)+\big((\Tbb-1)\,\Ubb(G\smallsetminus e)
-\Ubb(G)\big)\, r}
{1- (\Tbb+2)\,r+2\, r^2}
\]
if $e$ is not a bridge in $G$.
\end{thm}

\begin{proof}
The argument proving Theorem~\ref{thm:rationG} shows that the coefficient of
$\rho^n$ in $P(\Tbb,\rho)$ is $0$ for $n\ge 3$, and for $n\ge 2$ if $e$ is not a bridge, as
observed above. If $e$ is a bridge, 
\[
U_1=(T+1)^2\, \Bb(3)\, \Ubb(G\smallsetminus e) = \Tbb(\Tbb+1)^4\, \Ubb(G\smallsetminus e) 
=  \Tbb(\Tbb+1)^3\, U_0\quad, 
\]
since $G_1$ is obtained by replacing the central split of $e$, a bridge, with a $3$-banana.
By the same token, 
\[
V_1=(T+1)^2\, \Bb(2)\, \Ubb(G\smallsetminus e) = \Tbb(\Tbb+1)^3\, \Ubb(G\smallsetminus e) 
=  \Tbb(\Tbb+1)^2\, U_0\quad.
\]
By~\eqref{qe:inthep} we have
\begin{align*}
U_2 &= \Tbb(\Tbb+1)^3\, U_1 + (\Tbb-1)\Tbb (\Tbb+1)^2\, V_1 \\
&= \Tbb^3(\Tbb+1)^4(\Tbb+3)\, U_0\quad.
\end{align*}
On the other hand,
\begin{align*}
\Tbb(\Tbb+1)^2(\Tbb+2)\, U_1-2\,\Tbb^2(\Tbb+1)^4\, U_0 &= 
\Tbb^2(\Tbb+1)^5(\Tbb+2) \,U_0 -2\,\Tbb^2(\Tbb+1)^4\, U_0 \\
&=\Tbb^3(\Tbb+1)^4(\Tbb+3)\, U_0\quad.
\end{align*}
This verifies that the coefficient of $\rho^2$ in $P(\Tbb,\rho)$ (see~\eqref{eq:pofT})
equals~$0$ in this case as well.
Therefore, in all cases we have 
\[
P(\Tbb,\rho) = \Ubb(G)+(\Ubb(G_1)-\Tbb(\Tbb+1)^2 (\Tbb+2)\,\Ubb(G))\,\rho\quad.
\]
If $e$ is a bridge, the coefficient of $\rho$ in $P(\Tbb,\rho)$ is
\[
\Ubb(G_1)-\Tbb(\Tbb+1)^2 (\Tbb+2)\,\Ubb(G)=-\Tbb(\Tbb+1)^2\,\Ubb(G)
\]
since $\Ubb(G_1)=U_1=\Tbb(\Tbb+1)^3\, U_0=\Tbb(\Tbb+1)^3\, \Ubb(G)$ as we 
observed above.
Therefore
\[
P(\Tbb,\rho) = \Ubb(G)-\Tbb(\Tbb+1)^2\,\Ubb(G) \rho=(1-r)\,\Ubb(G)
\]
if $e$ is a bridge, and this gives the first formula.

If $e$ is not a bridge, splitting it into three and $3$-bananifying the central edge
gives, arguing as in the proof of Theorem~\ref{thm:rationG},
\[
\Ubb(G_1)=\Tbb(\Tbb+1)^3\, \Ubb(G)+(\Tbb-1)\Tbb(\Tbb+1)^2\, \Ubb(G\smallsetminus e)
\]
and therefore
\[
\Ubb(G_1)-\Tbb(\Tbb+1)^2 (\Tbb+2)\,\Ubb(G)=
-\Tbb(\Tbb+1)^2\,\Ubb(G)+(\Tbb-1)\Tbb(\Tbb+1)^2 \,\Ubb(G\smallsetminus e)\quad.
\]
It follows that the degree-$1$ term in $P(\Tbb,\rho)$ in this case is
\[
(-\Ubb(G)+(\Tbb-1)\,\Ubb(G\smallsetminus e))\,\Tbb(\Tbb+1)^2\, \rho\quad,
\]
and this completes the proof of the second formula.
\end{proof}

The fact that the formulas in Theorem~\ref{thm:rationG2} depend on $r=\Tbb(\Tbb+1)^2\rho$
explains why the specific examples worked out in Propositions~\ref{prop:GCgamman}
and~\ref{prop:GCgammanv} included powers of $\Tbb$ and $\Tbb+1$ as stated. We recover
these results in the next two examples.

\begin{example}
Define the polynomials $A_n(t)$ by the power series expansion
\[
\sum_{n\ge 0} A_n(t) r^n = \frac{1-r}{1-(2+t)r+2r^2}
\]
(cf.~Proposition~\ref{prop:both}).
Then the first formula in Theorem~\ref{thm:rationG2} reads
\[
\sum_{n\ge 0} \Ubb(G_n) \rho^n = \left(\sum_{n\ge 0} A_n(t) r^n\right)\cdot\Ubb(G) 
= \left(\sum_{n\ge 0} A_n(\Tbb) \Tbb^n (\Tbb+1)^{2n} \rho^n\right)\cdot\Ubb(G)\quad.
\]
Equivalently,
\begin{equation}\label{eq:GCGn}
\Ubb(G_n) = \Tbb^n (\Tbb+1)^{2n} A_n(\Tbb)\cdot \Ubb(G)\quad.
\end{equation}
If $G$ consists of a single edge, then with notation as in~\S\ref{MEcacI} we have 
$G_n=\Gamma_n$, and $\Ubb(G)=\Tbb+1$, therefore~\eqref{eq:GCGn} gives
\[
\Ubb(\Gamma_n)=\Tbb^n (\Tbb+1)^{2n+1} A_n(\Tbb)\quad,
\]
proving Proposition~\ref{prop:GCgamman} (in the form given in Proposition~\ref{prop:both}).
\qede\end{example}

\begin{example}
Now let $G$ be a $2$-banana, and let $e$ be one of its (two) edges.
The graph $G\smallsetminus e$ is a single edge. Therefore
\[
\Ubb(G)=\Tbb(\Tbb+1)\quad,\quad \Ubb(G\smallsetminus e) = \Tbb+1\quad,
\]
and the second formula in Theorem~\ref{thm:rationG2} states that
\begin{align*}
\sum_{n\ne 0} \Ubb(G_n)\,\rho^n &= 
\frac{\Tbb(\Tbb+1)+\big((\Tbb-1)(\Tbb+1)-\Tbb(\Tbb+1)\big) r}
{1- (\Tbb+2)\,r+2\, r^2} \\
&=\frac{\big(\Tbb- r\big)}{1- (\Tbb+2)\,r+2\, r^2}\cdot (\Tbb+1)
\end{align*}
With notation as in~\S\ref{MEcacI}, the graph $G_n$ (consisting of a
chain of $n+1$ circles) equals~$\Gamma'_{n+1}$, with two extra
valence-$2$ vertices on the first circle. That is,
\[
\Ubb(\Gamma'_n)=\frac{\Ubb(G_{n-1})}{(\Tbb+1)^2}\quad.
\]
Now, since $r=\Tbb(\Tbb+1)^2\rho$,
\[
\Ubb(G_{n-1})=\Tbb^{n-1}(\Tbb+1)^{2n-2}\cdot
\text{coeff.~of $r^{n-1}$ in the expansion of }
\frac{\big(\Tbb- r\big)(\Tbb+1)}{1- (\Tbb+2)\,r+2\, r^2}
\]
hence
\[
\frac{\Ubb(G_{n-1})}{(\Tbb+1)^2} = \Tbb^{n-1}(\Tbb+1)^{2n-3}\cdot
\text{coeff.~of $r^{n-1}$ in the expansion of }
\frac{\big(\Tbb- r\big)}{1- (\Tbb+2)\,r+2\, r^2}
\]
and therefore
\[
\Ubb(\Gamma'_n) = \Tbb^{n-1}(\Tbb+1)^{2n-3}\cdot
\text{coeff.~of $r^n$ in the expansion of }
\frac{r\big(\Tbb- r\big)}{1- (\Tbb+2)\,r+2\, r^2}\quad.
\]
This holds for $n\ge 1$; setting (as in~\S\ref{MEcacI}) the constant term of the
relevant series to $1$ amounts to adding $1$ to this rational function, and
\[
1+\frac{r\big(\Tbb- r\big)}{1- (\Tbb+2)\,r+2\, r^2}
=\frac{(1-r)^2}{1- (\Tbb+2)\,r+2\, r^2}
\]
verifying Proposition~\ref{prop:GCgammanv}, in the form given in 
Proposition~\ref{prop:both}.
\qede\end{example}

\begin{example}
As a final example, we will prove Proposition~\ref{prop:stars}, by induction on the 
number $s$ of rays. For $s=1$, the statement reproduces 
Proposition~\ref{prop:GCgammanv}; so we only need to prove the induction step,
and we may assume $s>1$.

To transition from $\Sigma^{s-1}_n$ to $\Sigma^s_n$, view $\Sigma^s_n$ as the 
graph obtained by adding a chain of $3$-bananas to one of the edges $e$ of the
central circle in $G=\Sigma^{s-1}_n$. 
\begin{center}
\includegraphics[scale=.3]{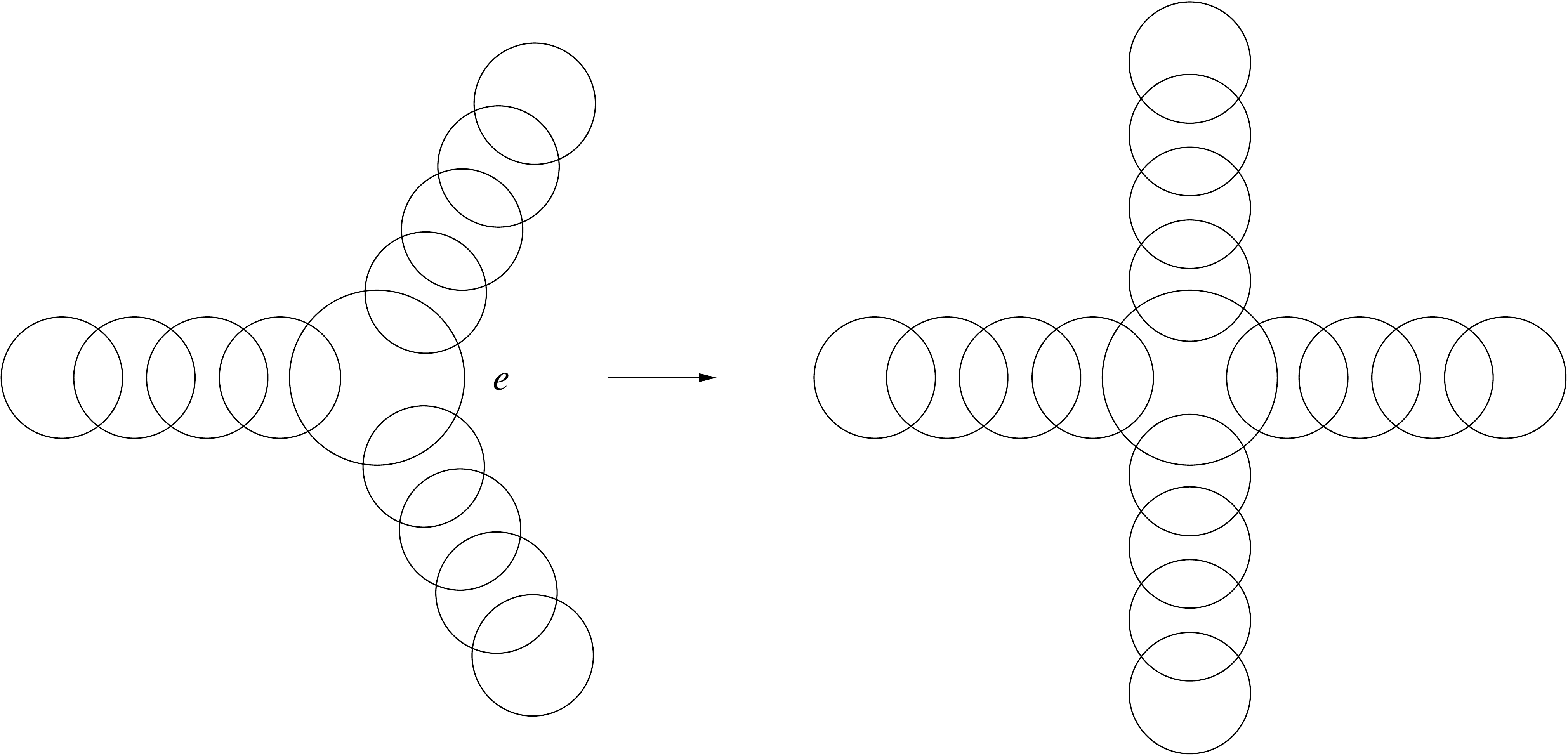}
\end{center}

Since $e$ is not a bridge, we can apply the second formula given in
Theorem~\ref{thm:rationG2}. We write it as follows:
\[
\frac{1-r}
{1- (\Tbb+2)\,r+2\, r^2}\,\Ubb(\Sigma^{s-1}_n)
+
\frac {r\,(\Tbb-1)} 
{1- (\Tbb+2)\,r+2\, r^2}\,\Ubb(\Sigma^{s-1}_n\smallsetminus e)\quad.
\]
The class $\Ubb(\Sigma^s_n)$ is the coefficient of $\rho^n$ in this expression
(i.e., $\Tbb^n(\Tbb+1)^{2n}$ times the coefficient of $r^n$).
We will deal with the two summands separately.
\begin{itemized}
\item
By induction, the first summand equals
\[
\frac{\Ubb(\Gamma_n)}{(\Tbb+1)} \cdot
\Tbb^{(s-1)n}(\Tbb+1)^{2(s-1)n-1} A_n(\Tbb)^{s-2} \cdot
\text{coeff.~of $r^{n+1}$ in } \frac{1-2r+((s-2)\Tbb-(s-3))r^2}{1-(\Tbb+2)r+2r^2}
\]
and $\Ubb(\Gamma_n)=\Tbb^n(\Tbb+1)^{2n+1} A_n(\Tbb)$, so this equals
\[
\Tbb^{sn}(\Tbb+1)^{2sn-1} A_n(\Tbb)^{s-1} \cdot
\text{coeff.~of $r^{n+1}$ in } \frac{1-2r+((s-2)\Tbb-(s-3))r^2}{1-(\Tbb+2)r+2r^2}\quad.
\]
\item
In the second summand, $\Sigma^{s-1}_n\smallsetminus e$ consists of 
a join of $s-1$ chains of $n$-circles, therefore its Grothendieck class
$\Ubb(\Sigma^{s-1}_n\smallsetminus e)$ is the $(s-1)$-st power of
$\Ubb(\Gamma_n)$, up to an appropriate factor of $(\Tbb+1)$ to account
for the fact that
$\Sigma^{s-1}_n\smallsetminus e$ has no valence-$2$ vertices and no external
edges. For example, here is a picture contrasting the join of $3$ graphs $\Gamma_4$
(on the left) with $\Sigma^3_4\smallsetminus e$ (on the right):
\begin{center}
\includegraphics[scale=.4]{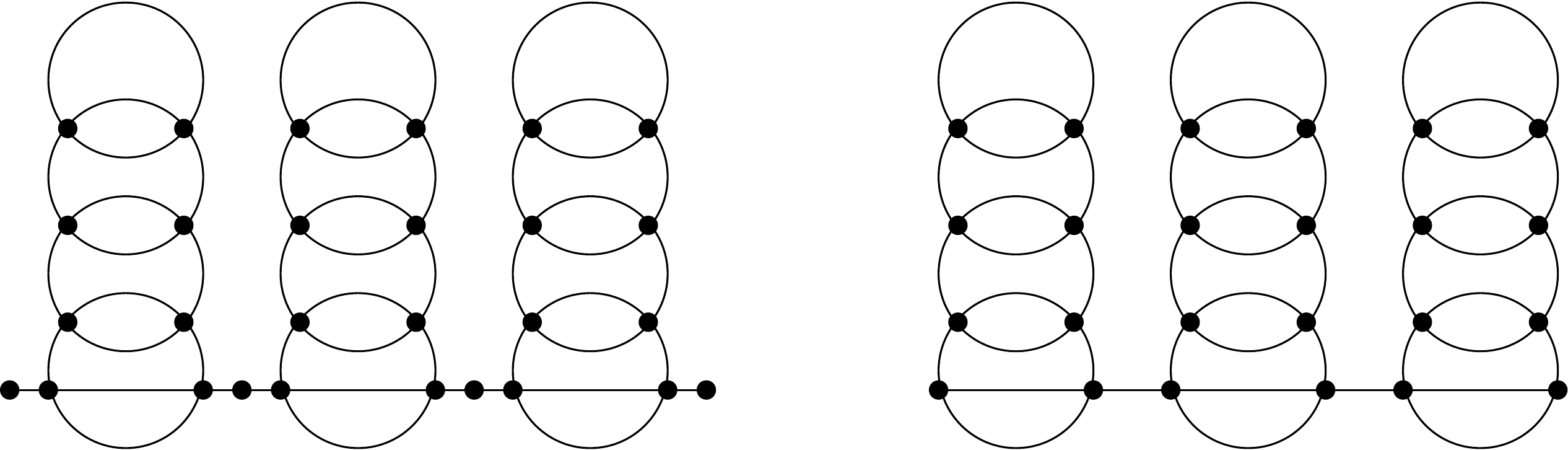}
\end{center}
It follows that
\[
\Ubb(\Sigma^{s-1}_n\smallsetminus e)=\frac{\Ubb(\Gamma_n)^{s-1}}{(\Tbb+1)^s}
=\Tbb^{(s-1)n}(\Tbb+1)^{2(s-1)n-1} A_n(\Tbb)^{s-1}\quad.
\]
Therefore, the second summand equals
\[
\Tbb^{(s-1)n}(\Tbb+1)^{2(s-1)n-1} A_n(\Tbb)^{s-1}\cdot
\text{coeff.~of $\rho^n$ in }
\frac {r\,(\Tbb-1)} {1- (\Tbb+2)\,r+2\, r^2}\quad.
\]
Now, the coefficient of $\rho^n$ equals $\Tbb(\Tbb+1)^2$ times the coefficient of $r^n$,
so this may be rewritten as
\[
\Tbb^{sn}(\Tbb+1)^{2sn-1} A_n(\Tbb)^{s-1}\cdot
\text{coeff.~of $r^n$ in }
\frac {r\,(\Tbb-1)} {1- (\Tbb+2)\,r+2\, r^2}
\]
or equivalently
\[
\Tbb^{sn}(\Tbb+1)^{2sn-1} A_n(\Tbb)^{s-1}\cdot
\text{coeff.~of $r^{n+1}$ in }
\frac {r^2\,(\Tbb-1)} {1- (\Tbb+2)\,r+2\, r^2}\quad.
\]
\end{itemized}
Putting the summands back together, we see that $\Ubb(\Sigma^s_n)$ equals
$\Tbb^{sn}(\Tbb+1)^{2sn-1} A_n(\Tbb)^{s-1}$ times the coefficient of $r^{n+1}$ in
{\small\[
\frac{1-2r+((s-2)\Tbb-(s-3))r^2}{1-(\Tbb+2)r+2r^2}
+\frac {r^2\,(\Tbb-1)} {1- (\Tbb+2)\,r+2\, r^2}
=\frac{1-2r+((s-1)\Tbb-(s-2))r^2}{1-(\Tbb+2)r+2r^2}
\]}
and this verifies the induction step, concluding the proof of Proposition~\ref{prop:stars}.
\qede\end{example}

\bigskip
\subsection*{Acknowkedgment} 
The first author acknowledges support from a Simons Foundation Collaboration Grant, 
award number 625561, and thanks the University of Toronto for hospitality.
The second author is partially supported by
NSF grant DMS-1707882, and by NSERC Discovery Grant RGPIN-2018-04937 
and Accelerator Supplement grant RGPAS-2018-522593, and by the Perimeter
Institute for Theoretical Physics. The third author worked on parts of this project
as summer undergraduate research at the University of Toronto.

\end{document}